\documentclass[12pt]{article}

\usepackage[english]{babel}
\usepackage{upref,amsfonts,amsxtra,latexsym,color}
\usepackage{amssymb,amsmath,amsthm,a4wide,epsf,mathrsfs,verbatim,hyperref}
\usepackage[all]{xy}

\newtheorem{theorem}{Theorem}[subsection]
\newtheorem{lemma}[theorem]{Lemma}
\newtheorem{corollary}[theorem]{Corollary}
\newtheorem{proposition}[theorem]{Proposition}

\theoremstyle{definition}
\newtheorem{definition}[theorem]{Definition}
\newtheorem{remark}[theorem]{Remark}
\newtheorem{example}[theorem]{Example}
\newtheorem{question}[theorem]{Question}

\numberwithin{equation}{section}
\numberwithin{theorem}{section}

\newcommand{\rank}{\mathrm{rank}}

\newcommand{\image}{\mathrm{Im}}

\newcommand{\Cu}{\mathrm{Cu}}
\newcommand{\CCu}{\mathbf{Cu}}

\newcommand{\cM}{{\cal M}}
\newcommand{\ep}{{\varepsilon}}

\newcommand{\cO}{{\cal O}}
\newcommand{\cT}{{\cal T}}

\newcommand{\cC}{{\mathcal C}}
\newcommand{\cZ}{{\mathcal Z}}
\newcommand{\C}{{\mathbb C}}
\newcommand{\Z}{{\mathbb Z}}
\newcommand{\N}{{\mathbb N}}

\newcommand{\T}{{\mathbb T}}

\newcommand{\cK}{{\cal K}}

\newcommand{\R}{{\mathbb R}}
\newcommand{\Cs}{{$C^*$-al\-ge\-bra}}

\newcommand{\ddv}{{\partial\mathrm{iv}}}
\newcommand{\dv}{{\mathrm{div}}}
\newcommand{\Dv}{{\mathrm{Div}}}
\newcommand{\rnk}{{\mathrm{rank}}}

\newcommand{\sh}{{$^*$-ho\-mo\-mor\-phism}}

\newcommand{\nprecsim}{{\operatorname{\hskip3pt \precsim\hskip-9pt |\hskip6pt}}}

\title{Divisibility properties for \Cs s}

\author{Leonel Robert$^*$ and Mikael R\o rdam\footnote{This research was supported by the Danish National Research Foundation (DNRF) through the Centre for Symmetry and Deformation}}

\date{}

\begin{document}
\maketitle
 
\begin{abstract} We consider three notions of divisibility in the Cuntz semigroup of a \Cs, and show how they reflect properties of the \Cs. We develop methods to construct (simple and non-simple) \Cs s with specific divisibility behaviour. As a byproduct of our investigations, we show that there exists a sequence $(A_n)$ of simple unital infinite dimensional \Cs s such that the product $\prod_{n=1}^\infty A_n$ has a character. 
\end{abstract}


\section{Introduction} A unital embedding of a matrix algebra $M_m(\C)$ into a unital \Cs{} $A$ can exist only if the equation $mx = [1_A]$ has a solution $x \in K_0(A)$. Thus, only \Cs s in which the class of the unit in $K_0$ is $m$-divisible admit a unital embedding of $M_m(\C)$.  Whereas all von Neumann algebras (with no central summand of type I$_n$ for $n$ finite) have this divisibility property for all $m$, the same is not true for \Cs s, even for the simple ones. \Cs s can fail to have non-trivial projections. Even if they have many projections, as in the real rank zero case,  one cannot expect to solve the equation $mx = [1_A]$ exactly in $K_0(A)$. This paper is concerned with different weaker notions of divisibility, phrased in terms of the Cuntz semigroup of the \Cs, and with how they relate to embeddability properties of the \Cs. Instead of solving the equation $mx = [1_A]$ for $x \in K_0(A)$, one should  look for less restrictive notions of divisibility. One can try, for example,  to solve the inequalities $mx \le \langle 1_A \rangle \le nx$ in the Cuntz semigroup of $A$ for fixed positive integers $m$ and $n$ (typically with $m < n$). We say that $A$ is $(m,n)$-divisible if one can solve this inequality. This is one of three divisibility properties we shall consider in this paper. We show that there is a full \sh{} from $CM_m(\C)$, the cone over $M_m(\C)$, into $A$ if and only if $A$ is $(m,n)$-divisible for some $n$. 

Let us mention three embedding problems that served as motivation for this paper. Let $A$ be a unital \Cs{} with no non-zero finite dimensional representations. Can one always find an embedding of some unital simple infinite dimensional \Cs{} into $A$? Can one always find an embedding of  $CM_2(\C)$  into $A$ whose image is full in $A$? Can one always find two positive mutually orthogonal full elements in $A$?  An affirmative answer to the former problem will imply an affirmative anwer to the second problem, which is known  as the ``Global Glimm Halving problem". An affirmative answer to the Global Glimm Halving problem will imply an affirmative answer to the last mentioned problem. We suspect that all three problems may have negative answers. 

The second and the third problem led us to consider two new notions of divisibility properties. In more detail, we say that $A$ is weakly $(m,n)$-divisible if there are elements $x_1, \dots, x_n$ in $\Cu(A)$ such that $mx_j \le \langle 1_A \rangle \le x_1+ \dots + x_n$. Weak divisibility measures the rank of $A$ in the sense that $A$ is weakly $(m,n)$-divisible for some $n$ if and only if $A$ has no non-zero representations of dimension $< m$. In particular, $A$ has no non-zero finite dimensional representations if and only if for every $m$ there is $n$ such that $A$ is weakly  $(m,n)$-divisible. We say that $A$ is $(m,n)$-decomposable if there are elements $y_1, \dots, y_m$ in $\Cu(A)$ such that $y_1+\cdots +y_m \le \langle 1_A \rangle \le ny_j$.  For a given $m$, $A$ is $(m,n)$-decomposable for some $n$ if and only if $A$ contains $m$ pairwise orthogonal, pairwise equivalent full positive elements.

It was shown in \cite{DadHirTomsWin} that there exists a simple unital infinite dimensional \Cs{} which does not admit a unital embedding of the Jiang-Su algebra $\cZ$. This answered in the negative a question posed by the second named author. It is implicit in \cite{DadHirTomsWin}  that this simple \Cs{} has bad divisibility properties, cf.\ Remark~\ref{rem:J-S}. This leads us to a useful observation, which loosely can be formulated as follows: if $A$ and $B$ are unital \Cs s, and if there is a unital \sh{} from $A$ to $B$, then the divisibility properties of $B$ are no worse than those of $A$. In other words, if $A$ has better divisibility properties than $B$, then you can not unitally embed $A$ into $B$. 

Comparability in the Cuntz semigroup is  concerned with the extent to which one can conclude that $x \le y$ if the ``size" of $x$ (e.g., measured in terms of states) is (much) smaller than the ``size" of $y$. Comparability
and divisibility are probably the two most fundamental properties of the Cuntz semigroup. Good comparability and divisibility properties are necessary and sufficient conditions in Winter's theorem, \cite{Win:Z-stable}, to conclude that a simple, separable, unital \Cs{} with locally finite nuclear dimension tensorially absorbs the Jiang-Su algebra. Also, good comparability and divisibility properties are both necessary and sufficient conditions to ensure that the Cuntz semigroup of a simple, separable, unital, exact \Cs{} $A$ is (naturally) isomorphic to ${\mathrm{Aff}}(T(A)) \sqcup V(A)$, cf.\ \cite{PerToms:recasting}, \cite{BrownPerToms}, and \cite{Elliott2008}. 

The existence of simple \Cs s with bad comparability properties was discovered by Villadsen, \cite{Vil:perforation}, in the mid 1990's. This discovery was the first indication that the Elliott conjecture could be false (in general), and it was also the first example of a simple \Cs{} exhibiting ``infinite dimensional" behaviour. Villadsen's example in \cite{Vil:perforation} has been generalized extensively by several authors (including Villadsen himself) to exhibit simple \Cs s with various kinds of unexpected behaviour, including many ways of failing to have good comparability properties. However, little work has been done to construct simple \Cs s with bad divisibility behaviour, and the literature does not contain systematic ways of producing such examples.  In this paper we show that there is a duality between comparability and divisibility (Lemma~\ref{lm:Z}), and we use this duality to construct examples of simple and non-simple \Cs s with bad divisibility behaviour. 

We use  Lemma~\ref{lm:Z} to obtain a result that concerns the structure of \Cs s that arise as the tensor product of a sequence of unital (simple non-elementary) \Cs s. Each such \Cs{} will of course have non-trivial central sequences. Dadarlat and Toms proved in \cite{DadToms:Z} that if the infinite tensor power $\bigotimes_{n=1}^\infty A$ of a fixed unital \Cs{} $A$ contains a unital copy of an AHS-algebra without characters, then it automatically absorbs the Jiang-Su algebra. It is not known if this condition always is satisfied, even when $A$ is simple and non-elementary. We show in Section~\ref{sec:divcomp} that $\bigotimes_{n=1}^\infty A$ has the Corona Factorization Property for every unital $A$ without characters (and in particular for every unital simple \Cs{} $A \ne \C$). In the other direction we give, in  Section~\ref{sec:obstruct}, an example of a sequence of simple unital infinite dimensional \Cs s whose tensor power, $\bigotimes_{n=1}^\infty A_n$, does not absorb (or admit an embedding of) the Jiang-Su algebra. 

Non-divisibility of a \Cs{} can be interpreted as a degree of inhomogeneity (or ``lumpiness") of the \Cs.  Simple \Cs s are sometimes thought of as being very homogeneous, as for example in \cite{KOS:homogeneity}. 
From this point of view it may at first be surprising that a simple infinite dimensional \Cs{} can fail to have good divisibility properties. We show that there exists a sequence $(A_n)$ of simple, unital, infinite dimensional \Cs s such that $\prod_{n=1}^\infty A_n$ (and also the associated ultrapowers of $(A_n)$) has a character. None of the \Cs s $A_n$ can have a character (being simple and not equal to $\C$), however we can show that they posses ``almost characters" as defined in Section~\ref{sec:ultrapowers}. 

In Section~\ref{sec:infinite} we consider what one might call ``super-divisibility", which leads to a (new) notion of infiniteness of positive elements (and which implies that a multiple of the given element is properly infinite). We use this to reformulate the Corona Factorization Property of semigroups considered in \cite{OPR:Cuntz}. We study variations of examples, originally due to Dixmier and Douady, and answer in this way two questions from \cite{KirRor:pi} in the negative: The sum of two properly infinite positive elements need not be properly infinite, and the multiplier algebra of a \Cs{} which has a properly infinite strictly positive element need not be properly infinite.

\section{Preliminaries} \label{sec:preliminary}
Let $A$ be a \Cs{} and let $\Cu(A)$ denote the Cuntz semigroup of $A$, i.e., the set of Cuntz equivalence classes of positive elements in $A\otimes \mathcal K$ endowed with suitable order and addition structures. Given a  positive element $a$ in $A\otimes \mathcal K$, we shall  denote by $\langle a\rangle$ the Cuntz class of $a$. In \cite{CoEllIv:cuntz}, Coward, Elliott and Ivanescu give  an alternative picture of the Cuntz semigroup where $\Cu(A)$ consists of suitable equivalence classes of  countably generated Hilbert $C^*$-modules over $A$. When using the Hilbert modules picture of $\Cu(A)$, we shall denote the equivalence class of a countably generated Hilbert module $H$ by $[H]$.

We present here some well-known definitions and facts about  the Cuntz semigroup. First of all, we shall frequently use the axioms of the category $\mathbf{Cu}$, of which $\Cu(A)$ is always an object (see \cite{CoEllIv:cuntz}).  An ordered abelian semigroup $S$ is an object in the category $\mathbf{Cu}$ if 
\begin{enumerate}
\item[(A1)]
every increasing sequence in $S$ has a supremum,
\item[(A2)]
for every $u\in S$ there exists a sequence $(u_i)$ in $S$ such that $u_i\ll u_{i+1}$ and $\sup_i u_i=u$,
\item[(A3)]
if $u'\ll u$ and $v'\ll v$, then $u'+v'\ll u+v$,
\item[(A4)]
if $(u_i)$ and $(v_i)$ are increasing sequences then $\sup_i u_i + \sup_i v_i=\sup_i (u_i+v_i)$.
\end{enumerate}
Recall that $u \ll v$ in $S$ if whenever $v = \sup_i v_i$ for some increasing sequence $(v_i)$ in $S$, then $u \le v_i$ for some $i$. An element $u \in S$ is called \emph{compact} if $u \ll u$. 

We also note the following two additional properties of the Cuntz semigroup of a \Cs which are note listed among the axioms of
$\mathbf{Cu}$. The first of them asserts that the Cuntz semigroup of a  \Cs{} almost has the Riesz Decomposition Property, and the second states that its order relation is almost the algebraic order. 
\begin{itemize}
\item[(P1)]
if $u'\ll u\leq v+w$, then there exist $v'$ and $w'$, with $v'\leq u,v$ and
$w'\leq u,w$, and such that  $u'\ll v'+w'$.
\item[(P2)] if $u'\ll u\leq  v$, then there exists $w$ such that $u'+w\leq v\leq u+w$.
\end{itemize}
For the proofs of these facts, see \cite[Proposition 5.1.1]{Robert2011}  for the first and \cite[Lemma 7.1 (i)]{RorWin:Z} for the second.

We will also make use of the sequential continuity with respect to inductive limits of the functor $\Cu(\cdot)$ proved in \cite{CoEllIv:cuntz} (see also the proof of \cite[Theorem 4.8]{Elliott2008}). It can be stated as follows:

\begin{proposition}[\cite{CoEllIv:cuntz}]
\label{prop:Culimits}
Let $A=\varinjlim (A_i,\varphi_{i,j})$ be a sequential inductive limit of \Cs s. 
\begin{enumerate}
\item For every $u\in \Cu(A)$ there exists an increasing sequence $(u_i)_{i=1}^\infty$
with supremum $u$ and such that each $u_i$ belongs to $\bigcup_j \image(\Cu(\varphi_{j,\infty}))$.

\item If $u,v\in \Cu(A_i)$ are such that $\Cu(\varphi_{i,\infty})(u)\leq \Cu(\varphi_{i,\infty})(v)$,
then for every $u'\ll u$ there exists $j$ such that $\Cu(\varphi_{i,j})(u')\leq \Cu(\varphi_{i,j})(v)$.
\end{enumerate}
\end{proposition}

\begin{remark}[The cone over a matrix algebra] \label{rem:CM}
Let $m$ be a positive integer, and let $CM_m(\C)$ denote the cone over $M_m(\C)$, i.e., the \Cs{} of all continuous functions $f \colon [0,1] \to M_m(\C)$ that vanish at $0$. 

For each $i,j=1,2, \dots, m$, let  $e_{ij}$ denote the $(i,j)$th matrix unit  in $M_m(\C)$, and denote by $e_{ij} \otimes \iota$ the function $t \mapsto te_{ij}$ in $CM_m(\C)$. Then $(e_{ii} \otimes \iota)_{i=1}^m$ are positive contractions in $CM_m(\C)$ which are pairwise equivalent and orthogonal. (We say here that two positive elements $a$ and $b$ are equivalent, denoted $a \sim b$, if $a = xx^*$ and $b = x^*x$ for some element $x$ in the ambient \Cs.) 

We recall the following well-known universal property of $CM_m(\C)$ (see for example \cite[Propositions 2.3 and 2.4]{RorWin:Z}): 
Let $A$ be any \Cs{} and let $a_1,a_2, \dots, a_m$ be positive contractions in $A$. Then there exists a \sh{} $\varphi \colon CM_m(\C) \to A$ satisfying $\varphi(e_{jj} \otimes \iota) = a_j$ if and only if $a_1,a_2, \dots, a_m$ are pairwise orthogonal and pairwise equivalent in $A$. 
\end{remark}

\noindent The following lemma is well-known:

\begin{lemma} \label{lm:comp}
Let $A$ be a \Cs{} and let $a,b_1,b_2,\dots,b_n$ be positive elements in $A$. Then:
\begin{enumerate}
\item 
$\langle a\rangle\leq \sum_{i=1}^n\langle b_i\rangle$ if and only if 
for each $\ep>0$ there exist $d_1,d_2,\dots,d_n\in A$ such that 
$(a-\ep)_+=\sum_{i=1}^n d_ib_id_i^*$.
\item $\sum_{i=1}^n\langle b_i\rangle \le \langle a\rangle$ if and only if for each $\ep > 0$ there exist mutually orthogonal positive elements $a_1,a_2, \dots, a_n$ in $\overline{aAa}$ such that $a_i \sim (b_i-\ep)_+$ for all $i$. 
\end{enumerate}
\end{lemma}

\begin{proof} (i). If $\langle a\rangle\leq \sum_{i=1}^n\langle b_i\rangle$, then $a \precsim b_1 \oplus b_2 \oplus \cdots \oplus b_n$, whence 
$$(a-\ep)_+ = d^*( b_1 \oplus b_2 \oplus \cdots \oplus b_n)d = \sum_{i=1}^n d_i^*b_id_i$$
for some $d = (d_1,d_2, \dots,d_n)^t \in M_{n,1}(A)$. The converse statement is trivial.

(ii). Suppose that  $\sum_{i=1}^n\langle b_i\rangle \le \langle a\rangle$. Then 
$(b_1-\ep)_+ \oplus (b_2-\ep)_+ \oplus \cdots \oplus (b_n -\ep)_+ = d^*ad$ for some $d = (d_1,d_2, \dots, d_n)$ in $M_{1,n}(A)$. Thus $d_j^*ad_i = 0$ if $j \ne i$ and $d_i^*ad_i = (b_i-\ep)_+$ for all $i$.
Put $a_i = a^{1/2}d_id_i^*a^{1/2}$. It is now straightforward to verify that the $a_i$'s are as desired. The converse statement is trivial.
\end{proof}

\noindent
Here is another lemma that we will use frequently:
\begin{lemma} \label{lm:cone-comp}
Let $a$ and $b$ be positive elements in $A\otimes \mathcal K$ with $\|a\|\leq 1$, and let $m\in \N$.
The following are equivalent:
\begin{enumerate}
\item  $m\langle a\rangle\leq \langle b\rangle$,
\item for each $\ep >0$ there exist mutually orthogonal positive elements $b_1,b_2,\dots,b_m$ in $\overline{b(A\otimes \mathcal K)b}$ such that  $\langle b_i\rangle=\langle(a-\ep)_+\rangle$
for all $i$.
\item  for each $\ep>0$ there exists a $^*$-homomorphism $\varphi\colon CM_m(\C)\to \overline{b(A\otimes \mathcal K)b}$ such that $\langle\varphi(e_{11} \otimes \iota)\rangle=\langle(a-\ep)_+\rangle$.
\end{enumerate}
\end{lemma}

\begin{proof}
The implications (iii) $\Rightarrow$ (ii) $\Rightarrow$ (i) are clear, cf.\ Remark~\ref{rem:CM} and Lemma~\ref{lm:comp}. Let us show that (i)
implies (iii). Let $\ep > 0$ be given. By Lemma~\ref{lm:comp} (ii) there are mutually orthogonal positive elements $b_1,b_2, \dots, b_m$ in $\overline{b(A\otimes \mathcal K)b}$ such that each $b_j$ is equivalent to $(a-\ep)_+$. By the universal property of the cone $CM_m(\C)$, see Remark \ref{rem:CM}, there is a \sh{} $\varphi \colon CM_m(\C) \to \overline{b(A\otimes \mathcal K)b}$ satisfying $\varphi(e_{jj} \otimes \iota) = b_j$. Hence (iii) holds.
\end{proof}

\section{Three divisibility properties} \label{sec:3div}

\subsection*{Definitions and basic properties}
\begin{definition} \label{def:divisibility}
Let $A$ be a \Cs{} and fix $u\in \Cu(A)$.  Let $m,n\ge 1$ be integers. Then:
\begin{enumerate}
\item 
$u$ is \emph{$(m,n)$-divisible} if for every $u' \in \Cu(A)$ with $u' \ll u$ there exists $x \in \Cu(A)$ such that $mx \le u$ and $u' \le nx$.

The least $n$ such that $u$ is $(m,n)$-divisible is denoted  by $\Dv_m(u,A)$, with $\Dv_m(u,A)=\infty$ if no such $n$ exists.

\item 
$u$ is \emph{$(m,n)$-decomposable} if for every $u' \in \Cu(A)$ with $u' \ll u$ there exist elements $x_1,x_2, \dots, x_m \in \Cu(A)$ such that $x_1+x_2 + \cdots + x_m \le u$ and $u' \le nx_j$ for all $j=1,2, \dots, m$.

The least $n$ such that $u$ is $(m,n)$-decomposable is denoted  by $\ddv_m(u,A)$, with  $\ddv_m(u,A)=\infty$ if no such $n$ exists.

\item
$u$ is \emph{weakly $(m,n)$-divisible} if for every $u' \in \Cu(A)$ with $u' \ll u$ there exist elements $x_1,x_2, \dots, x_n \in \Cu(A)$ such that $mx_j \le u$ for all $j=1,2, \dots, m$ and $u' \le x_1+x_2 + \cdots + x_n$. 

The least $n$ such that $u$ is weakly $(m,n)$-divisible is denoted  by $\dv_m(u,A)$, with $\dv_m(u,A)=\infty$ if no such $n$ exists.
\end{enumerate}
\end{definition}

\begin{remark} \label{rem:divisible-unital} In the case that $u$ in Definition~\ref{def:divisibility} is compact (e.g., when $A$ is unital and $u = \langle 1_A \rangle$), the conditions above read a little easier:
\begin{enumerate}
\item 
$u$ is \emph{$(m,n)$-divisible} if  there exists $x \in \Cu(A)$ such that $mx \le u \le nx$.
\item 
$u$ is \emph{$(m,n)$-decomposable} if there exist elements $x_1,x_2, \dots, x_m \in \Cu(A)$ such that $x_1+x_2 + \cdots + x_m \le u \le nx_j$ for all $j=1,2, \dots, m$. 
\item
$u$ is \emph{weakly $(m,n)$-divisible} if there exist elements $x_1,x_2, \dots, x_n \in \Cu(A)$ such that $mx_j \le u \le x_1+x_2 + \cdots + x_n$. 
\end{enumerate}
\end{remark}

\noindent The three divisibility properties above are related as follows:
\begin{proposition} \label{prop:a} 
Let $m,n\in \N$ and $u\in \Cu(A)$. Then
$$
\dv_m(u,A)  \leq \Dv_m(u,A), \quad
\ddv_m(u,A)  \leq \Dv_m(u,A),\quad
\dv_m(u,A)   \leq \ddv_m(u,A)^m.
$$
\end{proposition}

\begin{proof} 
The two first inequalities are clear (take $x_i = x$ in both cases).

To prove the last inequality, suppose that  $u$ is $(m,n)$-decomposable. We show that $u$ is weakly $(m,n^m)$-divisible. Let $u'\ll u$ and find $u''$ such that $u'\ll u''\ll u$. There exist elements $x_1, \dots, x_m$ in $\Cu(A)$ such that $\sum_{i=1}^m x_i\leq u$ and $u''\leq nx_i$ for all $i$. We proceed to find elements
$$\tilde{y}(i_1,\dots, i_k), \, y(i_1, \dots, i_k) \in \Cu(A), \quad k=1, \dots, m, \, i_j=1,\dots, n,$$
satisfying
\begin{itemize}
\item[(a)] $\tilde{y}(i_1, \dots, i_k) \ll y(i_1, \dots, i_k)$,
\item[(b)] $\tilde{y}(i_1, \dots, i_{k-1}) \ll \sum_{i=1}^n y(i_1, \dots,i_{k-1}, i)$ if $k \ge 2$, and $u' \ll \sum_{i=1}^n y(i)$, 
\item[(c)] $\tilde{y}(i_1, \dots, i_{k-1}) \le \sum_{i=1}^n \tilde{y}(i_1, \dots,i_{k-1}, i)$ if $k \ge 2$, and $u' \le \sum_{i=1}^n \tilde{y}(i)$, 
\item[(d)] $y(i_1,\dots, i_k) \le x_k$,
\item[(e)] $y(i_1,\dots,i_{k-1}, i_k) \le y(i_1,\dots, i_{k-1})$ if $k\ge 2$, and $y(i) \le u''$.
\end{itemize}

The elements above are constructed inductively after $k$ using the following fact:
\begin{itemize}
\item[($*$)] if $x' \ll x \le nz$ in $\Cu(A)$, then there exist $y_1, \dots, y_n \in \Cu(A)$ such that $x' \ll \sum_{i=1}^n y_i$, $y_i \le x$, and $y_i \le z$,
\end{itemize}
which follows from Property (P1) of the Cuntz semigroup stated in the previous section.

Take first $k=1$. The existence of $y(i)$, with $i=1, \dots, n$, satisfying (b), (d) and (e) follows from ($*$) applied to $u' \ll u'' \le nx_1$. The existence of $\tilde{y}(i) \ll y(i)$ satisfying (a) and (c) then follows from Axiom (A2) of the Cuntz semigroup from the previous section. Assume that $2 \le k \le m$ and that $\tilde{y}(i_1,i_2, \dots, i_{k-1})$ and $y(i_1,i_2, \dots, i_{k-1})$ have been found. The existence of $y(i_1, \dots, i_{k-1},i)$, with $i=1, \dots, n$, satisfying (b), (d) and (e) follows from ($*$) applied to 
$$\tilde{y}(i_1, \dots, i_{k-1}) \ll y(i_1, \dots, i_{k-1}) \le nx_k.$$ 
(To see that the latter inequality holds, note that $y(i_1, \dots, i_{k-1}) \ll u''$, which follows by repeated use of (e).) The existence of 
$\tilde{y}(i_1,i_2, \dots, i_k)$ satisfying (a) and (c) follows from Axiom (A2). 

We claim that the $n^m$ elements $\big(y(i_1,\dots, i_m)\big)$ witness the weak $(m,n^m)$-divisibility of $u$. Indeed, it follows from (d) and (e) that $y(i_1,\dots, i_m) \le x_j$ for  all $j =1, \dots, m$, whence
$$m \! \cdot \! y(i_1,\dots, i_m)  \le x_1+x_2 + \cdots + x_m \le u.$$
It follows from (b) and (c) that the sum of the elements $y(i_1,\dots, i_m)$ is greater than or equal to $u'$. 
%
\end{proof}

\noindent If any of the divisibility numbers $\Dv_m(u,A)$, $\ddv_m(u,A)$, and $\dv_m(A)$ is less than $m$, then $u$ (or a multiple of $u$) must be properly infinite, as shown below. We shall pursue this and related questions in more detail in Section~\ref{sec:infinite}. 

\begin{proposition}  \label{prop:propinf} 
Let $A$ be a \Cs{} and let $u\in \Cu(A)$.
\begin{enumerate}
\item If $u$ is properly infinite, then $\Dv_m(u,A) =1$ for all integers $m \ge 1$. 
\item If $1 \le n < m$ are integers and if $u$ is either $(m,n)$-divisible, $(m,n)$-decomposable or weakly $(m,n)$-divisible, then $nu$ is properly infinite, i.e., $nu=2nu$. 
\item If $1 \le n < m$ are integers and if $u$ is compact and $(m,n)$-divisible, then $u$ is properly infinite. 
\end{enumerate}
\end{proposition}

\begin{proof} 
(i). If $u$ is properly infinite, then $mu \le u$ for all $m$, whence $\Dv_m(u,A) =1$ .

(ii). Assume that $u$ is weakly $(m,n)$-divisible and take $u'\ll u$. Then there exist $x_1, \dots, x_n$ such that $mx_i\leq u$ for all $i$, and $u'\leq\sum_{i=1}^n  x_i$. Thus,
\[
mu'\leq \sum_{i=1}^n mx_i\leq nu.
\] As this holds for all $u'\ll u$, we get  $((m-n) +n)u = mu \leq nu$. This entails that $(k(m-n)+n)u \leq nu$ for all positive integers $k$, whence $\ell u \leq nu$ for all $\ell \ge n$. In particular, $2nu \le nu$, which implies that $nu$ is properly infinite.

Next, suppose that $u$ is $(m,n)$-decomposable and let $u'\ll u$. Then 
there exist $x_1, \dots, x_m$ such that 
$\sum_{i=1}^m x_i\leq u$ and  $u'\leq nx_i$ for all $i$. Thus,
\[
mu'\leq n\sum_{i=1}^m x_i\leq nu.
\] 
Arguing as before, we conclude that $nu$ is properly infinite.

Finally note that if $u$ is $(m,n)$-divisible, then it is both $(m,n)$-decomposable and weakly $(m,n)$-divisible, whence $nu$ is properly infinite. 

(iii). Since  $\Dv_m(u,A) = n < m$ and $u\ll u$, there exists $x$ such that $mx \le u \le nx$. Arguing as above this implies that  $nx$ is properly infinite. 
\end{proof}

\begin{remark} \label{rem:image}
By functoriality, each \sh{} $\varphi\colon A\to B$ between \Cs s $A$ and $B$ induces a morphism  $\Cu(\varphi)\colon \Cu(A)\to \Cu(B)$ which preserves order, addition, and the relation of compact containment. Thus, for each $u\in \Cu(A)$, and with $v=\Cu(\varphi)(u)$, we have:
$$
\Dv_m(v,B)  \le  \Dv_m(u,A),\quad
\ddv_m(v,B) \le \ddv_m(u,A),\quad
\dv_m(v,B) \le \dv_m(u,A). 
$$
In particular, if $A$ and $B$ are unital \Cs s, and if $\Dv_m(\langle 1_B \rangle,B) > \Dv_m(\langle 1_A \rangle, A)$ for some $m$ (or if the corresponding inequality holds for one of the other two divisibility numbers), then one can not find a unital embedding of $A$ into $B$. Divisibility numbers thus serve as an obstruction for embedding a unital \Cs{} with nice divisibility properties into a unital \Cs{} with less nice divisibility properties. 
\end{remark}

\noindent
The three divisibility properties behave well with respect to inductive limits thanks to the sequential continuity of the functor $\Cu(\cdot)$:

\begin{proposition} \label{prop:inductive-limits}
Let $A=\varinjlim (A_i,\varphi_{i,j})$ be a sequential inductive limit of \Cs. Let $u\in \Cu(A_1)$ and, for each $i$, denote by
$u_i\in \Cu(A_i)$ and $u_\infty\in \Cu(A)$ the images of $u$ in $\Cu(A_i)$ and $\Cu(A)$, respectively. Then:
$$
\Dv_m(u_\infty,A) \leq \inf_i \Dv_m(u_i,A_i),\qquad
\ddv_m(u_\infty,A) \leq \inf_i \ddv_m(u_i,A_i), 
$$
$$
\dv_m(u_\infty,A)  \leq  \inf_i\dv_m(u_i,A_i). 
$$
If $u$ is compact (i.e., if $u\ll u$), then the above inequalities are equalities.
\end{proposition}
\begin{proof}
We will only prove the statements above in the former case; the proofs for the two other cases are similar. 

The inequalities $\Dv_m(u_\infty,A) \leq \Dv_m(u_i,A_i)$, with $i=1,2,\dots$, follow from 
Remark~\ref{rem:image}. Suppose now that $u$ is compact. Set $\Dv_m(u_\infty,A)=n$. Then there exists 
$x\in \Cu(A)$ such that $mx\le u_\infty \le nx$.
By Proposition \ref{prop:Culimits} (i) and compactness of $u_\infty$, it follows that $x$ is the image of some $y\in \Cu(A_i)$ for some $i$. By Axiom (A2) of the Cuntz semigroup and by compactness of $u_\infty$ there exists $y'\ll y$ in $\Cu(A_i)$ such that
$u_\infty \le n\, \Cu(\varphi_{i,\infty})(y')$. Since the $u_i$'s are compact,  Proposition \ref{prop:Culimits} (ii) implies that there exists $j>i$ such that
$$m \, \Cu(\varphi_{i,j})(y') \le u_j \le n \, \Cu(\varphi_{i,j})(y').$$
Thus $u_{j}$ is $(m,n)$-divisible in $\Cu(A_{j})$.
\end{proof}

\begin{definition} \label{def:divA}
Let $A$ be a $\sigma$-unital \Cs. Then $A$ contains a strictly positive element. This element represents a class in $\Cu(A)$, which is independent of the choice of the strictly positive element, and which we shall denote by $\langle A \rangle$. If $A$ is unital, then $\langle A \rangle = \langle 1_A \rangle$. We shall write $\dv_m(A)$, $\ddv_m(A)$, and $\Dv_m(A)$ for $\dv_m(\langle A \rangle, A)$, $\ddv_m(\langle A \rangle, A)$, and $\Dv_m(\langle A \rangle, A)$, respectively.
\end{definition}

\noindent
If $A$ and $B$ are unital \Cs s such that there exist unital \sh s $A \to B$ and $B \to A$, then, by Remark~\ref{rem:image}, we must have
$$\Dv_m(A) = \Dv_m(B), \quad \ddv_m(A) = \ddv_m(B), \quad \dv_m(A) = \dv_m(B)$$
for all $m$. This applies in particular to the situation where $A$ is any unital \Cs{} and $B = A \otimes D$ for some unital \Cs{} $D$ which has a character. In general, if $D$ is any unital \Cs, possibly without characters, the divisibility numbers associated with $A \otimes D$ are smaller than or equal to those of $A$.

\subsection*{Examples and remarks}

\noindent
Let us first examine the divisibility numbers for matrix algebras:
\begin{example} \label{rem:matrix}
Let $m \ge 2$ and $k \ge 2$ be integers. 
Using that 
\[
\big(\Cu(M_k(\C)), \langle 1 \rangle \big) \cong \big(\{0,1,2,3, \dots, \infty\}, k\big), 
\]
an elementary algebraic argument yields that 
\[
\dv_m(M_k(\C)) = \ddv_m(M_k(\C))=\Dv_m(M_k(\C)),\] 
and
\begin{equation} \label{eq:matrix}
\Dv_m(M_k(\C)) = \begin{cases} 
\left\lceil \frac{k}{\left\lfloor \frac{k}{m}\right\rfloor}\right\rceil, & \text{if} \; m \leq k, \\
\infty, & \text{if} \; m > k.
\end{cases}
\end{equation}
Here $\lceil\cdot \rceil$ and $\lfloor\cdot \rfloor$ are the ``ceiling" and ``floor" functions. 
In particular,  $\Dv_m(M_k(\C))=m$ if and only if $m \!\mid \! k$, and $\Dv_m(M_k(\C))=m+1$
if $m \! \nmid \! k$ and $m(m-1)\leq k$.
\end{example}

\begin{definition}[The rank of a \Cs] \label{def:rank}
Let $A$ be a \Cs. Let $\rnk(A)$ denote the smallest positive integer $n$ for which $A$ has an irreducible representation on a Hilbert space of dimension $n$, and set $\rnk(A) = \infty$ if $A$ has no finite dimensional (irreducible) representation. 
\end{definition}

\noindent Note that $\rnk(A) = 1$ if and only if $A$ has a character.  We remind the reader about the following classical result due to Glimm:

\begin{proposition}[Glimm] \label{prop:Glimm}
Let $A$ be a (not necessarily unital) \Cs. Then there is a non-zero \sh{} $CM_n(\C) \to A$ if and only if $A$ admits at least one irreducible representation on a Hilbert of dimension $\ge n$. 
\end{proposition}

\noindent
It follows from Remark \ref{rem:image} that $\dv_m(A)$, $\ddv_m(A)$, and $\Dv_m(A)$ are greater than or equal to $\Dv_m(M_n(\C))$ if $\rnk(A) = n$. In particular, these three quantities are infinite when $m > \rnk(A)$. 

\begin{example}[Simple \Cs s]  \label{ex:simple}
If $A$ is a simple, unital, infinite dimensional \Cs, then $\Dv_m(A)$, $\ddv_m(A)$, and $\dv_m(A)$ are finite for all positive integers $m$. Indeed, by the assumption that $A$ is infinite dimensional, it follows that there is a non-zero $x \in \Cu(A)$ such that $mx \le \langle 1_A \rangle$. As every simple unital \Cs{} is algebraically simple, it follows that $\langle 1_A \rangle \le nx$ for some positive integer $n$, i.e., $\langle 1_A \rangle$ is $(m,n)$-divisible. Hence $\Dv_m(A) \le n$, which entails that also $\ddv_m(A) \le n$ and $\dv_m(A) \le n$. 
\end{example}

\begin{example}
The dimension drop \Cs{} $Z_{p,q}$, associated with the positive integers $p$ and $q$, is defined to be
$$Z_{p,q} = \{f \in C([0,1],M_p \otimes M_q \mid f(0) \in M_p \otimes \C 1_q, \; \; f(1) \in \C 1_p \otimes M_q \}.$$
Note that $\rnk(Z_{p,q}) = \min\{p,q\}$. 
It was shown in \cite[Lemma 4.2]{Ror:Z} (and its proof) that $\Dv_m(Z_{m,m+1}) = m+1$. 
By Remark~\ref{rem:image}, it follows that if $Z_{m,m+1}$ maps unitally into $A$, then $\Dv_m(A) \le m+1$.   Moreover,  as shown in \cite[Proposition 5.1]{RorWin:Z},  if $A$ is a unital \Cs{} of stable rank one, then $\Dv_m(A) \le m+1$ if and only if $Z_{m,m+1}$ maps  unitally into $A$. 
\end{example}

\begin{remark}[Almost divisibility] 
The property ``almost divisibility" of a \Cs{} is expressed by saying that $\Dv_m(A) \le m+1$ for all integers $m \ge 1$. If every dimension drop algebra $Z_{m,m+1}$ maps unitally into $A$, or if the Jiang-Su algebra maps unitally into $A$, then $A$ is almost divisible. 
\end{remark}

\begin{remark}[Non-embeddability of the Jiang-Su algebra] \label{rem:J-S}
It was shown in \cite{DadHirTomsWin} that there is a simple unital infinite dimensional nuclear \Cs{} $A$ such that the dimension drop \Cs{} $Z_{3,4}$, and hence the Jiang-Su algebra $\cZ$, do not embed unitally into $A$. The divisibility properties of $A$ were not explicitly mentioned in \cite{DadHirTomsWin}, but it is easily seen (using Lemma \ref{lm:Z}, that is paraphrased from \cite[Lemma 4.3]{Ror:Z}) that $\Dv_3(A) > 4$.
We shall in Section \ref{sec:obstruct} give further examples of simple unital infinite dimensional \Cs s where the divisibility numbers attain non-trivial values.
\end{remark}

\begin{remark}[Real rank zero \Cs s] \label{ex:rr0}
It was shown in \cite[Proposition 5.7]{PerRor:AF} that if $A$ is a unital \Cs{} of real rank zero then
$A$ admits a unital embedding of a finite dimensional \Cs{} of rank at least $n$ if and only if $\rnk(A) \ge n$. Combining this with Remark \ref{rem:matrix} we see that $\Dv_m(A) \le m+1$ whenever $A$ is a unital \Cs{} of real rank zero and with $\rnk(A) \ge m(m-1)$. 
In particular, every unital \Cs{} $A$ of real rank zero and with $\rnk(A) = \infty$ is almost divisible.
\end{remark}

\noindent
Kirchberg considered in \cite{Kir:Abel} a covering number of a unital \Cs{} $B$. 
Let us recall the definition: 

\begin{definition}[Kirchberg] \label{def:cov}
Let $m\in \N$. The covering number of a unital \Cs{} $B$, denoted by $\mathrm{cov}(B,m)$,  is the least positive integer $n$ such that there exist finite dimensional \Cs s $F_1, F_2, \dots, F_n$ with $\rnk(F_i) \ge m$, $^*$-homomorphisms  $\varphi_i\colon CF_i\to B$, and $d_1,d_2,\dots,d_n\in B$ such that
$1_B=\sum_{i=1}^n d_i^*\varphi_i(1_{F_i} \otimes \iota) d_i$.
\end{definition}

\noindent
Kirchberg's covering number $\mathrm{cov}(B,m)$ relates to our $\dv_m(B)$ as follows.
\begin{proposition} \label{prop:cov}
Let $B$ be a unital \Cs{} and let $m$ be a positive integer.
\begin{enumerate}
\item $\mathrm{cov}(B,m)$ is the least $n$ for which there exist $x_1,x_2, \dots, x_n \in \Cu(B)$ such that
\begin{equation} \label{eq:cov}
x_i\leq \langle 1_B \rangle\leq x_1+x_2+ \cdots + x_n, \qquad x_i=\sum_{j=i}^{k_i} m_{ij}y_{ij}
\end{equation}
for some 
integers $m_{ij}\geq m$, some positive integers $k_i$, and some $y_{ij} \in \Cu(A)$.
\item $\mathrm{cov}(B,m) \le \dv_m(B)\leq (2m-1)\, \mathrm{cov}(B,m)$.
\end{enumerate}
\end{proposition}

\begin{proof}
(i). Assume that $n \ge \mathrm{cov}(B,m)$ and let $F_i$, $\varphi_i \colon CF_i \to B$, and $d_i \in B$ be as in Definition \ref{def:cov}. Write $F_i = \bigoplus_{j=1}^{k_i} M_{m_{ij}}(\C)$ with $m_{ij} \ge m$. Let $e^{(ij)}$ be a one-dimensional projection in $M_{m_{ij}}(\C)$. It then follows from Lemma~\ref{lm:comp} that the elements
$$x_i = \langle \varphi_i(1_{F_i} \otimes \iota) \rangle, \qquad y_{ij} = \langle \varphi_i(e^{(ij)} \otimes \iota) \rangle$$
satisfy the relations in \eqref{eq:cov}. 

Suppose, conversely, that $n \ge 1$ is chosen such that there are elements $x_i$ and $y_{ij}$ in $\Cu(B)$ satisfying \eqref{eq:cov}. Put $F_i = \bigoplus_{j=1}^{k_i} M_{m_{ij}}(\C)$. By the assumption that $\sum_{j=i}^{k_i} m_{ij}y_{ij} \le \langle 1_B \rangle$ it follows from Lemma~\ref{lm:comp} (ii) that there are mutually orthogonal positive elements $a_{ijr}$ in $B$, where $1 \le i \le n$, $1 \le j \le k_i$, $1 \le r \le m_{ij}$, such that $\langle a_{ijr} \rangle = y_{ij}$. We can further assume that the $r$ positive elements $a_{ij1}, \dots, a_{ijr}$ are pairwise equivalent. It then follows from the universal property of the cone over a matrix algebra (see Remark~\ref{rem:CM}) that there are \sh s $\varphi_i \colon CF_i \to B$ such that $\langle \varphi_i(e^{(ij)} \otimes \iota) \rangle= y_{ij}$, where $e^{(ij)}$ is a one-dimensional projection in the summand $ M_{m_{ij}}(\C)$ of $F_i$. The existence of $d_i \in B$ with $1_B=\sum_{i=1}^n d_i^*\varphi_i(1_{F_i} \otimes \iota) d_i$ follows from Lemma~\ref{lm:comp} (i). Thus $\mathrm{cov}(B,m) \le n$.

(ii). To prove the first inequality, assume that $\dv_m(B) = n < \infty$ and take $y_1, \dots, y_n$ such that $my_j \le \langle 1_B \rangle \le y_1 + \cdots + y_n$. Then \eqref{eq:cov} holds with $k_i = 1$ and $x_i = my_i$.

Assume next that $\mathrm{cov}(B,m) = n < \infty$, and find elements $x_i$ and $y_{ij}$ satisfying the relations in \eqref{eq:cov}. Upon replacing $y_{ij}$ with an integral multiple of $y_{ij}$ we can assume that $m \le m_{ij} < 2m$ for all $i$ and $j$. Let $z_{ik}$, $1 \le k \le 2m-1$, be the sum of a suitable subset of the $y_{ij}$'s such that
$\sum_{k=1}^{2m-1} z_{ik} =  \sum_{j=i}^{k_i} m_{ij}y_{ij} = x_i$. The $(2m-1)n$ elements $(z_{ik})$ will then witness that $\dv_m(B) \le (2m-1)n$.
\end{proof}

\section{The asymptotic divisibility numbers} \label{sec:stable-div}
One can collect the sequence of divisibility numbers $\big(\Dv_m(A)\big)_{m=2}^\infty$ of a unital \Cs{} $A$ into a single divisibility number as follows:
$$
\Dv_*(A) =  \liminf_{m \to \infty} \frac{\Dv_m(A)}{m}.
$$
In a similar way one can define $\ddv_*(A)$ and $\dv_*(A)$. Propositions \ref{prop:almdiv} and \ref{lm:dv-stable} below hold verbatim for those quantities as well. However, to keep the exposition bounded, we only treat the case of ``$\Dv$".

It follows from Proposition~\ref{prop:propinf} that $\Dv_*(A) = 0$ if and only if $A$ is properly infinite and that $\Dv_*(A) \ge 1$ if $A$ is not properly infinite. 

\begin{proposition} \label{prop:almdiv} Let $A$ be a unital \Cs. 
\begin{enumerate}
\item  $\Dv_m(A) \le m \, \Dv_*(A)+1$ for all integers $m \ge 2$.
\item $\Dv_*(A) = \lim_{m \to \infty} \Dv_m(A)/m$ (the limit always exists, but is possibly equal to $\infty$).
\item If $A$ is not properly infinite, then $\Dv_*(A) = 1$ if and only if $\Dv_m(A) \le m+1$ for all integers $m \ge 2$.
\end{enumerate}
\end{proposition}

\noindent
It follows from Proposition~\ref{prop:propinf} and from (iii) above, that $A$ is almost divisible if and only if $\Dv_*(A) \le 1$ (i.e., if and only if $\Dv_*(A) = 0$ or $\Dv_*(A) = 1$). 

\begin{proof}  (i). If $\Dv_*(A) = \infty$ there is nothing to prove. Assume that $1 \le \Dv_*(A) < \infty$. 
Let $m \ge 2$ be given. Let $L$ be the smallest integer strictly greater than $m \, \Dv_*(A)$. We show that $\Dv_m(A) \le L$. Choose $\alpha > 1$ and a positive integer $r_0$ such that
$$\alpha \, \frac{r_0+1}{r_0} \, m \, \Dv_*(A) \le L.$$
By the definition of $\Dv_*(A)$ there is $k \ge r_0 m$ such that $\ell := \Dv_k(A)  \le \alpha k \, \Dv_*(A)$.
Take $x \in \Cu(A)$ such that $kx \le \langle 1_A \rangle \le \ell x$. 
Write $k = rm + d$, with $0 \le d < m$ and $r \ge r_0$. Also, write $\ell = tr-d'$,  with $0 \le d' < r$ and $t \ge 1$. Put $y = rx \in \Cu(A)$. Then $my \le \langle 1_A \rangle \le ty$. With $\lceil \, \cdot \, \rceil$ denoting the ceiling function, we have
\begin{eqnarray*}
\Dv_m(A)  \; \le \; t & = & \lceil \frac{\ell}{r} \rceil \\
&=& \Big\lceil \frac{\ell}{k-d} \, m \Big\rceil \\
& \le & \Big\lceil \alpha \, \frac{k}{k-d} \, m \, \Dv_*(A) \Big\rceil \\
& \le & \Big\lceil \alpha \, \frac{r+1}{r} \, m \, \Dv_*(A) \Big\rceil \\
& \le & \Big\lceil \alpha \, \frac{r_0+1}{r_0} \, m \, \Dv_*(A) \Big\rceil \; \le \; L.
\end{eqnarray*}

(ii). It follows from (i) that 
$$
\frac{\limsup_{m \to \infty} \Dv_m(A)}{m} \; \le \; \Dv_*(A) \; = \; 
\frac{\liminf_{m \to \infty} \Dv_m(A)}{m},
$$
and so the claims follows.

(iii). The ``if" part is trivial, and the ``only if" part follows from (i).
\end{proof}

\noindent We proceed to discuss how $\Dv_*( \, \cdot \,)$ behaves under forming matrix algebras:

\begin{proposition} \label{lm:dv-stable} Let $A$ be a unital \Cs.
\begin{enumerate}
\item $\Dv_*(M_n(A)) \le \Dv_*(A)$ for all integers $n \ge 2$.
\item If $\Cu(A)$ is almost unperforated, then $\Dv_*(M_n(A)) = \Dv_*(A)$ for all integers $n \ge 2$.
\end{enumerate}
\end{proposition}

\begin{proof} 
(i) follows from Remark~\ref{rem:image} (as $A$ embeds unitally into $M_n(A)$).
 
(ii). Assume that $\Cu(A)$ is almost unperforated. We show first that 
\begin{equation} \label{eq:3}
\Dv_*(A) \le \frac{n+1}{n-1} \, \Dv_*(M_n(A))
\end{equation}
 for all $n \ge 2$. To see this take any integer $m \ge 2$, and use Proposition \ref{prop:almdiv} (i) to see that $\ell := \Dv_m(M_n(A)) \le m \, \Dv_*(M_n(A)) + 1$. Write $m = r (n+1) + d$ and $\ell = t(n-1) - d'$, where $r$ and $t$ are positive integers, $0 \le d < n+1$, and $0 \le d' < n-1$. 

Identify $\Cu(M_n(A))$ with $\Cu(A)$ in the canonical way, where $\langle 1_{M_n(A)} \rangle \in \Cu(M_n(A))$ is identified with $n \langle 1_A \rangle$. Under this identification we can find $x \in \Cu(A)$ such that $mx \le n \langle 1_A \rangle \le \ell x$. In particular,
$$(n+1)rx \le n \langle 1_A \rangle \le (n-1) tx,$$
which by the assumption that $\Cu(A)$ is almost unperforated implies that $rx \le \langle 1_A \rangle \le tx$. This shows that
\begin{eqnarray*}
\frac{\Dv_r(A)}{r} \; \le \; \frac{t}{r} &=& r^{-1}\Big\lceil \frac{\ell}{n-1} \Big\rceil \; \le \; r^{-1} \big( \frac{\ell}{n-1} +1\big) \\
& \le & r^{-1}\Big( \frac{m \, \Dv_*(M_n(A)) +1}{n-1} +1 \Big) \\
& \le &  \frac{n+1}{n-1} \, \Dv_*(M_n(A)) + r^{-1}  \frac{n}{n-1}  \, \Dv_*(M_n(A)) + r^{-1}\frac{n}{n-1}.
\end{eqnarray*}
Now, $r \to \infty$ as $m \to \infty$, and so \eqref{eq:3} follows by letting $m$ tend to infinity.

To complete the proof of (ii), take $n \ge 2$. By (i) and \eqref{eq:3} we have:
$$\Dv_*(A) \; \le \; \frac{kn+1}{kn-1} \, \Dv_*(M_{nk}(A)) \; \le \; \frac{kn+1}{kn-1} \, \Dv_*(M_{n}(A))$$
for all $k \ge 1$, which shows that $\Dv_*(A) \le \Dv_*(M_n(A))$.
\end{proof}

\noindent We have previously remarked that $\Dv_m(A) = \infty$ whenever $m > \rnk(A)$. It follows that $\Dv_*(A) = \infty$ whenever $\rnk(A) < \infty$, i.e., whenever $A$ admits a non-zero finite dimensional representation.

\begin{remark} It can happen that $\Dv_*(M_n(A)) < \Dv_*(A)$. Take for example $A$ such that $M_n(A)$ is properly infinite, but $A$ itself is not properly infinite, cf.\ \cite{Ror:simple}. Then $\Dv_*(M_n(A)) = 0$ and $\Dv_*(A) \ge 1$. 

It is an important open problem if $\Dv_*(A) \le 1$ (i.e., if $A$ is almost divisible) for every
(simple) unital infinite dimensional \Cs{} $A$ for which $\Cu(A)$ is almost unperforated. 
\end{remark}

\section{Finite-, infinite-, and $\omega$-divisibility}

\noindent The property that any of the divisibility numbers $\Dv_m(A)$, $\ddv_m(A)$, and $\dv_m(A)$ is finite, when $A$ is a unital \Cs, has interpretations in terms of structural properties of the \Cs{} $A$. We have already noted that the divisibility numbers always are finite when $A$ is a simple \Cs, and the corresponding structural properties of the \Cs{} are, as we shall see, trivially satisfied for simple \Cs s. The correct definition of ``finite divisibility" in the non-unital case is what we call $(m,\omega)$-divisibility as defined below. 
 
\begin{definition} \label{def:omega-divisibility}
Let $A$ be a \Cs, let $u\in \Cu(A)$, and let $m$ be a positive integer. Then:
\begin{enumerate}
\item 
$u$ is \emph{$(m,\omega)$-divisible} if for all $u' \in \Cu(A)$ with $u' \ll u$ there exists $x \in \Cu(A)$  such that $mx \le u$ and $u' \le n x$ for some positive integer $n$.
\item 
$u$ is \emph{$(m,\omega)$-decomposable} if for all $u' \in \Cu(A)$ with $u' \ll u$ there exist elements $x_1,x_2, \dots, x_m \in \Cu(A)$ such that $x_1+x_2 + \cdots + x_m \le u$ and $u' \le n x_j$  for some positive integer $n$ and for all $j$.
\item
$u$ is \emph{weakly $(m,\omega)$-divisible} if for all $u' \in \Cu(A)$ with $u' \ll u$ there exist elements $x_1,x_2, \dots, x_n$ in  $\Cu(A)$ such that $mx_j \le u$ for all $j$ and $u' \le x_1+x_2+ \cdots + x_n$. 
\end{enumerate}
\end{definition}

\begin{remark} \label{rem:omega-div}
If $u$ in Definition~\ref{def:omega-divisibility} is compact, then $u$ is $(m,\omega)$-divisible, $(m,\omega)$-de\-com\-po\-sab\-le, respectively, weakly $(m,\omega)$-divisible if and only if $\Dv_m(u,A) < \infty$, $\ddv_m(u,A) < \infty$, respectively, $\dv_m(u,A) < \infty$, cf.\ Remark~\ref{rem:divisible-unital}.
\end{remark}

\noindent
In the next result we express $(m,\omega)$-divisibility in terms of structural properties of the \Cs.
Part (iii) is almost contained in \cite{Kir:Abel} (see  \cite[Definition 3.1]{Kir:Abel} and 
\cite[Remark 3.3 (7)]{Kir:Abel} and compare with Definition~\ref{def:cov} and Proposition~\ref{prop:cov}). Recall the definition of the rank of a \Cs{} from Definition~\ref{def:rank}.

\begin{theorem} \label{thm:finite-div}
Let $A$ be a $\sigma$-unital \Cs{} and let $e$ be a strictly positive element of $A$. (If $A$ is unital,  we can take $e$ to be the unit of $A$.) Put $u = \langle e \rangle = \langle A \rangle$. 
\begin{enumerate}
\item
$u$ is $(m,\omega)$-divisible if and only if for every $\ep>0$ there exists a \sh{} $\varphi\colon CM_m(\C)\to A$ such that $(e-\ep)_+$  belongs to the closed two-sided ideal generated by the image of $\varphi$.
\item
 $u$ is $(m,\omega)$-decomposable if and only if for every $\ep>0$ there exist mutually orthogonal positive elements $b_1,b_2,\dots,b_m$ in $A$ such that $(e-\ep)_+$ belongs to the closed two-sided ideal generated by $b_i$  for each $i$. 
\item The following are equivalent:
\begin{enumerate}
\item $u$ is weakly $(m,\omega)$-divisible,
\item $\rnk(A)\geq m$,
\item there exist \sh s $\varphi_i\colon CM_m(\C)\to A$, $i=1,2,\dots, n$, for some $n$, such that $(e-\ep)_+$ belongs to the closed two-sided ideal generated by the union of the images of the $\varphi_i$'s. 
\end{enumerate}
\end{enumerate}
\end{theorem}

\begin{proof}
(i). Let us assume that $u$ is $(m,\omega)$-divisible. Let $\ep>0$. Find $x\in \Cu(A)$ and a positive integer $n$ such that $mx\leq u$ and $\langle(e-\ep/2)_+\rangle \leq nx$. Choose a positive element $a$ in $A \otimes \cK$ such that $x = \langle a \rangle$, and choose $\eta > 0$ such that $\langle(e-\ep)_+\rangle \leq n \langle (a-\eta)_+ \rangle$.
By Lemma~\ref{lm:cone-comp} there exists $\varphi\colon CM_m(\C)\to A$ such that $\langle \varphi(e_{11}\otimes \iota) \rangle = \langle (a-\eta)_+ \rangle$. Then $\langle (e-\ep)_+ \rangle \le n \langle \varphi(e_{11}\otimes \iota) \rangle$ which implies that $(e-\ep)_+$ belongs to the closed two-sided ideal generated by the image of $\varphi$.

Suppose conversely that for every  $\ep>0$ there exists  $\varphi\colon CM_m\to A$ such that $(e-\ep)_+$ is in the closed two-sided ideal generated by $\varphi(e_{11}\otimes\iota)$.  Set $\langle \varphi(e_{11}\otimes\iota) \rangle=x$. Then $mx\leq u$ by Lemma~\ref{lm:cone-comp}, while $\langle(e-2\ep)_+ \rangle \leq n x$ for some positive integer $n$. This shows that $u$ is $(m,\omega)$-divisible.

(ii). ``Only if". Let $\ep>0$ and suppose that $b_1, b_2, \dots, b_m$ in $A$ exist with the stipulated properties.  Set $\langle b_j \rangle=x_j \in \Cu(A)$. Then 
$$
x_1+x_2 + \cdots + x_m = \langle b_1+b_2 + \cdots + b_m \rangle \le u.
$$
Since $(e-\ep)_+$ belongs to the closed two-sided ideal generated by $b_j$,  $\langle (e-2\ep)_+ \rangle \leq n x_j$ for some integer $n \ge 1$. It follows 
that $u$ is $(m,\omega)$-decomposable.

``If". If $u = \langle e \rangle$ is $(m,\omega)$-decomposable and if $\ep > 0$, then there are positive elements $a_1, a_2, \dots, a_m$ in $A \otimes \cK$ such that $\langle a_1 \rangle+\langle a_2 \rangle+ \cdots +\langle a_m \rangle \le u$ and $\langle (e-\ep/2)_+ \rangle \le n \langle a_j \rangle$ for some positive integer $n$. Choose $\eta > 0$ such that $\langle (e-\ep)_+ \rangle \le n \langle (a_j-\eta)_+ \rangle$ for all $j$. By Lemma~\ref{lm:comp} (ii) there are pairwise orthogonal positive elements $b_1,b_2, \dots, b_m$ in $A$ such that $b_j \sim (a_j - \eta)_+$. Then the closed two-sided ideal generated by $b_j$ contains $(e-\ep)_+$ for each $j$.

(iii).  (a) $\Rightarrow$ (b). Assume that $u$ is weakly $(m,\omega)$-divisible. Suppose that $A$ has an irreducible representation $\pi \colon A \to B(\C^k) = M_k(\C)$ of finite positive dimension $k$.
Then $\pi$ is necessarily surjective. Since $(m,\omega)$-divisibility is preserved by \sh s (cf.\ Remark \ref{rem:image}), we conclude that
$M_k(\C)$ is weakly $(m,\omega)$-divisible. But then $k\geq m$, cf.\ Example~\ref{rem:matrix}. Hence (b) holds. 

 (b) $\Rightarrow$ (c). Assume that (b) holds. Let $(\varphi_i)_{i \in \mathbb{I}}$ be the family of all non-zero \sh s $\varphi_i \colon CM_m(\C) \to A$ and let $I$ be the closed two-sided ideal in $A$ generated by the images of all $\varphi_i$'s. Thus each $\varphi_i$ maps $CM_m(\C)$ into $I$. We claim that $I=A$. Assume, to reach a contradiction, that $I \ne A$. By the assumption that $\rnk(A) \ge m$, all irreducible representations of $A/I$ have dimension at least $m$. It follows from Glimm's lemma (Proposition~\ref{prop:Glimm}) that there is a non-zero \sh{} $CM_m(\C) \to A/I$, which by projectivity lifts to a \sh{} $\varphi \colon CM_m(\C) \to A$. But the image of $\varphi$ is not contained in $I$, which is a contradiction.

For each finite subset $F$ of $\mathbb{I}$ consider the closed two-sided ideal $I_F$ of $A$ generated by $\bigcup_{i \in F} \varphi_i(CM_m(\C))$. Then $A$ is the closure of the union of the upwards directed family of ideals $(I_F)$. Hence, for each $\ep > 0$, there is a finite subset $F$ of $\mathbb{I}$ such that $(e-\ep)_+$ belongs to $I_{F}$. Thus (c) holds.

 (c) $\Rightarrow$ (a). Assume that (c) holds. Set $z_i=\langle \varphi_i(e_{11}\otimes \iota) \rangle$ for $i=1,2,\dots, n$. Then $mz_i \le u$ for all $i$. Moreover, $(a-\ep)_+$ belongs to the algebraic ideal generated by the $n$ elements $\varphi_i(e_{11} \otimes \iota)$, whence $\langle (a-\ep)_+ \rangle \le \sum_{j=1}^n n_jz_j$ for suitable positive integers $n_j$. Put $N = \sum n_j$ and let $x_1, x_2, \dots, x_N$ be a listing of the elements $z_1, \dots, z_n$, with $z_j$ repeated $n_j$ times. Then $mx_j \le u$ and $(a-\ep)_+ \le x_1+x_2 + \dots + x_N$. This shows that $u$ is weakly $(m,\omega)$-divisible.
\end{proof}

\noindent The theorem above can be simplified in the case where $u$ is compact, and in particular in the case where $A$ is unital:

\begin{corollary} \label{cor:finite-div}
Let $A$ be a unital \Cs, and let $m$ be a positive integer. Then:
\begin{enumerate}
\item
$\Dv_m(A) < \infty$ if and only if there exists a \sh{} $\varphi\colon CM_m(\C)\to A$ whose image is full in $A$. 
\item
$\ddv_m(A) < \infty$ if and only if there exist full, pairwise orthogonal positive elements $b_1,b_2,\dots,b_m$ in $A$.
\item The following are equivalent:
\begin{enumerate}
\item $\dv_m(A) < \infty$, 
\item $\rnk(A)\geq m$,
\item there exist \sh s $\varphi_i\colon CM_m(\C)\to A$, $i=1,2,\dots, n$ for some $n$, such that the union of their images is full in $A$.
\end{enumerate}
\end{enumerate}
\end{corollary}

\noindent Propositions \ref{prop:Glimm} (Glimm), Proposition \ref{prop:a}, and Corollary~\ref{cor:finite-div} (i) immediately imply:

\begin{corollary} \label{cor:simple}
 If $A$ is a unital, infinite dimensional, simple \Cs, then the three divisibility numbers $\Dv_m(A)$, $\ddv_m(A)$, and $\dv_m(A)$ are finite  for every integer $m \ge 1$.
\end{corollary}

\noindent Let us also note what it means to have infinite $\dv_m( \, \cdot \,)$ numbers:

\begin{corollary} \label{cor:char}
 Let $A$ be a unital \Cs. 
\begin{enumerate}
\item $A$ admits a character if and only if $\dv_2(A) = \infty$. 
\item $A$ admits a finite dimensional representations if and only if $\dv_m(A) = \infty$ for some integer $m \ge 2$. 
\end{enumerate}
\end{corollary}

\begin{remark}[The Global Glimm Halving Problem]
Glimm's lemma (Proposition~\ref{prop:Glimm}) says that there exists a non-zero \sh{} from $CM_n(\C)$ into a \Cs{} $A$ if and only if $A$ admits an irreducible representation of dimension at least $n$. It is not known how ``large" one can make the image of such a \sh. 
In particular, it is not known for which \Cs s $A$ one can find a \sh{} $CM_n(\C) \to A$ whose image is full in $A$ (i.e., the image is not contained in any proper closed two-sided ideal in $A$). For $n=2$ this problem is known as the ``Global Glimm Halving Problem" (see \cite{BlanKir:Glimm}, \cite{BlanKir:pi3} and \cite{KirRor:pi2}). A unital \Cs{} $A$ is said to have the \emph{Global Glimm Halving Property} if there is a \sh{} $CM_2(\C) \to A$ with full image.

More specifically, one can ask if any (unital) \Cs, which admits no finite dimensional representation, satisfies the Global Glimm Halving Property. 
In view of Corollary~\ref{cor:finite-div}, this problem for unital \Cs s $A$ may be restated as follows: Does $\Dv_2(A) = \infty$ imply that $\dv_m(A) = \infty$ for some positive integer $m$? For a non-unital \Cs{} $A$, the one can restate the problem in the following way: Does $\Dv_2(A) = \infty$ imply that $\langle A \rangle$ fails to be $(m,\omega)$-divisible for some positive integer $m$. 

It is shown in \cite{KirRor:pi2} that if $A$ is a weakly purely infinite \Cs{}, then $A$ is purely infinite if and only if all hereditary sub-\Cs s of $A$ have the Global Glimm Halving Property. (It is easy to see that the rank of any weakly purely infinite \Cs{} is infinite.) It is an open problem if all weakly purely infinite \Cs s are purely infinite. 
\end{remark}

\begin{remark} Let $A$ be a unital \Cs.
It follows from Proposition \ref{prop:a} (and also from Corollary~\ref{cor:finite-div}) that
$$
\Dv_m(A) < \infty \implies \ddv_m(A) < \infty \implies \dv_m(A) < \infty
$$
for all positive integers $m$. None of the two reverse implications hold in general. 

For each integer $m \ge 2$ consider the Bott projection $p$ in $C(S^{2m}) \otimes \cK$. This projection has (complex) dimension $m$ and it has no non-trivial sub-projections. The unital \Cs{} $A = p(C(S^{2m})\otimes \cK)p$ is a homogeneous \Cs{} of rank $m$. Hence $\dv_m(A) < \infty$. Suppose that $\ddv_m(A) < \infty$. Then, by Corollary~\ref{cor:finite-div}~(ii), there would exist full, pairwise orthogonal, positive elements $b_1, \dots, b_m$ in $A$. This would entail that each $b_j$ is one-dimensional in each fiber of $A$, and hence that $f_j.b_j$ is a one-dimensional projection for some $f_j \in C(S^{2m})$. But this contradicts the fact that $p$ has no proper subprojections. 

To see that $\ddv_m(A) < \infty$ does not imply $\Dv_m(A) < \infty$,  consider the \Cs{} $B = C(S^2) \otimes \cK$, and take the (one-dimensional) Bott projection $p \in B$ and a trivial $(m-1)$-dimensional projection $q \in B$ such that $p$ and $q$ are orthogonal. Put $A = (p+q)B(p+q)$. It follows from a $K$-theoretical argument that $p+q$ cannot be written as the sum of $m$ pairwise orthogonal and equivalent projections (because $[p+q]$ is not divisible by $m$ in $K_0(A)$). The unit of $A$ can be written as the sum of $m$ (necessarily full) projections, so $\ddv_m(A) < \infty$. Assume that $\Dv_m(A) < \infty$. Then, by Corollary~\ref{cor:finite-div}~(i), there is a \sh{} $\varphi \colon CM_m(\C) \to A$ whose image is full in $A$.  As explained in Remark~\ref{rem:CM}, this entails that there exist full, pairwise orthogonal, pairwise equivalent, positive elements $b_1,b_2, \dots, b_m$ in $A$. Arguing as in the paragraph above, we can assume that each $b_j$ is in fact a projection. But that contradicts the fact that $1_A = p+q$ is not the sum of $m$ pairwise equivalent projections. 
\end{remark}

\section{Divisibility and comparability} \label{sec:divcomp}

Let $A$ and $B$ be \Cs s. Then there is a natural bi-additive map
$$\Cu(A) \times \Cu(B) \to \Cu(A \otimes B), \qquad (x,y) \mapsto x \otimes y,$$
defined as follows: If $x = \langle a \rangle$ and $y = \langle b \rangle$ with $a$ a positive element in $A \otimes \cK$ and $b$ a positive element in $B \otimes \cK$, then $x \otimes y = \langle a \otimes b \rangle$, where we identify $(A \otimes \cK) \otimes (B \otimes \cK)$ with $A \otimes B \otimes \cK$. Note that $x_1 \otimes y_1 \le x_2 \otimes y_2$ if $x_1 \le x_2$ and $y_1 \le y_2$.

Part (i) of the following result was (implicitly) proved in \cite[Lemma 4.3]{Ror:Z}, and was used to prove that $\Cu(A \otimes \cZ)$ is almost unperforated for all unital \Cs s $A$. 

\begin{lemma}  \label{lm:Z}
Let $A$ and $B$ be unital \Cs s and let $1 \le m < n$ be integers.
\begin{enumerate}
\item  Let $x,y \in \Cu(A)$ be such that $nx \le my$. If $B$ is $(m,n)$-divisible, then $x \otimes \langle 1_B \rangle \le y \otimes \langle 1_B \rangle$. 
\item Let $x_1,x_2, \dots, x_m,y \in \Cu(A)$ be such that $nx_j \le y$ for all $j$. If $B$ is $(m,n)$-de\-com\-po\-sa\-ble, then 
$$ (x_1  + x_2 + \cdots + x_m) \otimes \langle 1_B \rangle \le y \otimes \langle 1_B \rangle.$$
\item Let $x,y_1,y_2, \dots, y_n \in \Cu(A)$ be such that $x \le m y_j$ for all $j$. If $B$ is weakly $(m,n)$-di\-vi\-si\-ble, then
$$x \otimes \langle 1_B \rangle \le (y_1  + y_2 + \cdots + y_n) \otimes \langle 1_B \rangle.$$
\end{enumerate}
\end{lemma}

\begin{proof} 
(i). Take $z \in \Cu(B)$ such that $mz \le \langle 1_B \rangle \le nz$. Then
$$x \otimes \langle 1_B \rangle \le x \otimes nz = nx \otimes z \le my \otimes z = y \otimes mz \le
y \otimes \langle 1_B \rangle.$$

(ii). Take $z_1, z_2, \dots, z_m \in \Cu(B)$ such that $z_1+z_2+ \cdots + z_m \le \langle 1_B \rangle \le nz_j$. Then
\begin{eqnarray*}
(x_1  + x_2 + \cdots + x_m) \otimes \langle 1_B \rangle &\le& 
x_1 \otimes nz_1 + x_2 \otimes nz_2 + \cdots + x_m \otimes nz_m \\
&=& nx_1 \otimes z_1 + nx_2 \otimes z_2 + \cdots + nx_m \otimes z_m \\
& \le & y \otimes (z_1 + z_2 + \cdots + z_m) \\
& \le & y \otimes  \langle 1_B \rangle.
\end{eqnarray*}

(iii). Take $z_1, z_2, \dots, z_n \in \Cu(B)$ such that $mz_j \le \langle 1_B \rangle \le z_1+z_2 + \cdots + z_n$. Then
\begin{eqnarray*}
(y_1  + y_2 + \cdots + y_n) \otimes \langle 1_B \rangle &\ge& 
y_1 \otimes mz_1 + y_2 \otimes mz_2 + \cdots + y_n \otimes mz_n \\
&=& my_1 \otimes z_1 + my_2 \otimes z_2 + \cdots + my_n \otimes z_n \\
& \ge & x \otimes (z_1 + z_2 + \cdots + z_n) \\
& \ge & x \otimes  \langle 1_B \rangle.
\qedhere
\end{eqnarray*}
\end{proof}

\noindent The lemma above can loosely be paraphrased as follows: Good divisibility properties of $B$ ensure good comparability properties of $A \otimes B$, and bad comparability properties of $A \otimes B$ entail bad divisibility properties of $B$. 

We proceed to show that infinite tensor products of (suitable) unital \Cs s cannot have very bad comparability properties. 


\begin{lemma}  \label{lm:limit1}
Let $\big(A_k\big)_{k=1}^\infty$ be a sequence of unital \Cs s such that $N:=\sup_k \dv_2(A_k) < \infty$. Then
$$\dv_m( \bigotimes_{k=1}^\infty A_k) \le N^n,$$
for all integers $m \ge 2$, where $n$ is any integer satisfying $2^n \ge m$. 
\end{lemma}

\begin{proof} Take $n$ such that $2^n \ge m$. For each $k$, find $z^{(k)}_i \in \Cu(A_k)$, $i=1,2,\dots,N$, such that $2z^{(k)}_i \le \langle 1_{A_k} \rangle \le \sum_{i=1}^N z^{(k)}_i$. 
Given a multi-index $(i_1,i_2,\dots,i_n)\in \{1,\dots,N\}^n$, put 
$$
z_{i_1,i_2,\dots,i_n} = z^{(1)}_{i_1} \otimes z^{(2)}_{i_2} \otimes \cdots \otimes z^{(n)}_{i_n} \in \Cu\Big(\bigotimes_{k=1}^n A_k\Big).
$$
Then 
\begin{align*}
m\cdot z_{i_1,i_2,\dots,i_n}   \le   2^n z_{i_1,i_2,\dots,i_n} & =  (2z^{(1)}_{i_i}) \otimes (2z^{(2)}_{i_2}) \otimes \cdots \otimes (2z^{(n)}_{i_n}) \\
& \le \langle 1_{A_1} \otimes 1_{A_2} \otimes \cdots \otimes 1_{A_n} \rangle \\ 
&\le   \Big(\sum_{i=1}^N z^{(1)}_i\Big) \otimes \Big(\sum_{i=1}^N z^{(2)}_i\Big) \otimes \cdots \otimes 
\Big(\sum_{i=1}^N z^{(n)}_i\Big) \\
& = \sum_{i_1,i_2,\dots,i_n=1}^N z_{i_1,i_2,\dots,i_n}.\qedhere
\end{align*}
\end{proof}

\noindent
Recall that a \Cs{} $A$ has the Corona Factorization Property if and only if all full projections in $\cM(A\otimes \cK)$ are properly infinite. 
It was shown in \cite{OPR:Cuntz} that if $A$ is a separable \Cs, then $A$ and all its closed two-sided ideals have the Corona Factorization Property if and only if for every integer $m \ge 2$, and for all $x',x,y_1,y_2, \dots$ in $\Cu(A)$ such that $x' \ll x$ and $x \le my_j$ for all $j$,  one has $x' \le y_1+y_2 + \cdots + y_N$ for some integer $N \ge 1$. 

\begin{proposition} Let $\big(A_k\big)_{k=1}^\infty$ be a sequence of unital \Cs s such that $$\sup_k \dv_2(A_k) < \infty.$$ It follows that the \Cs{} $\bigotimes_{k=1}^\infty A_k$ and all its closed two-sided ideals have the Corona Factorization Property.
In particular, if $A$ is  a unital $C^*$-algebra without characters then $\bigotimes_{k=1}^\infty A$ and all its closed two-sided ideals have the Corona Factorization Property.
\end{proposition}

\begin{proof} 
Put $B =\bigotimes_{k=1}^\infty A_k$, and for each $n \ge 1$ put $B_n = \bigotimes_{k=1}^n A_k$ and  $D_n = \bigotimes_{k=n+1}^\infty A_k$. We shall view $(B_n)_{n=1}^\infty$ as an increasing sequence of sub-\Cs s of $B$ such that $\bigcup_{n=1}^\infty B_n$ is dense in $B$, and we shall identify $B$ with $B_n \otimes D_n$ for all $n$.

Let $m \ge 1$ be an integer and let $x',x,y_1,y_2,y_3, \dots$ in $\Cu(B)$ be such that $x' \ll x$ and $x \le my_j$ for all $j$. By Lemma~\ref{lm:limit1} there is a positive integer $N$ such that $\dv_m(D_n) \le N$ for all $n$. We show that $x' \le y_1+y_2 + \cdots + y_N$. This will prove that $B$ has the Corona Factorization Property. 

Repeated use of 
Proposition \ref{prop:Culimits} (i) and (ii) shows that there exists a positive integer $n$, and elements $x'', y_1',y_2', \dots, y_N'$ in $\Cu(B_n)$ such that 
$$x' \le x'' \otimes \langle 1_{D_n} \rangle \le x, \quad y_j' \otimes \langle 1_{D_n} \rangle \ll y_j \; \; \text{in} \;  \Cu(B); \qquad x'' \ll my_j'  \; \; \text{in} \; \Cu(B_n), \qquad  $$
where $x \mapsto x \otimes \langle 1_{D_n} \rangle$ denotes the canonical embedding $\Cu(B_n) \to \Cu(B)$. We can now apply Lemma~\ref{lm:Z} (iii) to deduce that 
$$x' \le x'' \otimes \langle 1_{D_n} \rangle \le (y_1' + y_2' + \cdots + y_N') \otimes \langle 1_{D_n} \rangle 
\le y_1+y_2+ \cdots + y_N$$
as desired.
\end{proof}

\section{Obstructions to Divisibility} \label{sec:obstruct}

\noindent
A trivial obstruction to (weak) divisibility of a \Cs{} is its rank: $\dv_m(A)<\infty$ if and only if  $m\leq \rnk(A)$ (by Corollary \ref{cor:finite-div} (iii)). In this section 
we shall discuss ways of obtaining homogeneous \Cs{}s with large rank and large weak divisibility constant. We use these techniques to construct unital simple \Cs s with large weak divisibility constants. 

We remark first that Lemma~\ref{lm:Z} provides non-trivial obstructions to divisibility in  $B$. Indeed, it follows by that lemma that if there exists a unital \Cs{} $A$ and $x,y \in \Cu(A)$ such that $nx \le my$ but $x \otimes \langle 1_B \rangle \nleq y \otimes \langle 1_B \rangle$, then $\Dv_m(B) > n$. Similarly, if there exist $x_1, \dots, x_m,y$ in $\Cu(A)$ such that $nx_j \le y$ for all $j$ while 
$$
(x_1 +x_2 + \cdots + x_m) \otimes \langle 1_B \rangle \nleq y \otimes \langle 1_B \rangle,
$$
then $\ddv_m(B) > n$. Finally, if there exist $x,y_1,\dots, y_n \in \Cu(A)$ such that $x \le my_j$ for all $j$ while 
$$
x \otimes \langle 1_B \rangle \nleq (y_1+y_2 + \cdots + y_n) \otimes \langle 1_B \rangle,
$$
then $\dv_m(B) > n$. 

We introduce below another way to obtain bad divisibility behavior:

\begin{lemma} \label{lm:2-div}
Let $u,v\in \Cu(A)$ be compact elements. If $\Dv_2(u+v,A)\leq N$ then there exist $x_1,x_2,\dots,x_{N}$ in $\Cu(A)$ such that $2x_i\leq v$ for all $i$ and 
\begin{align*}
2v\leq v+(2N+1)u+\sum_{i=1}^{N} 2x_i.
\end{align*}
\end{lemma}

\begin{proof}
By assumption there exists $x$ such that $2x\leq u+v \leq Nx$. By compactness of $u+v$ we can find $x'' \ll x' \ll x$ such that $u+v \le Nx''$. Since $x' \ll x \le u+v$, it follows from Property (P1) of the Cuntz semigroup (see Section~\ref{sec:preliminary}) (leaving $u$ unchanged) that there exists $v_1$ such that 
$$x' \le u+v_1, \qquad v_1 \le x, \qquad v_1 \le v.$$ 
As $x'' \ll u+v_1$ there is $v_1' \ll v_1$ such that $x'' \le u+v_1'$. Apply (P2) to the relation $v_1' \ll v_1 \le v$ to obtain $v_2$ satisfying
$$v_1'+v_2 \le v \le v_1 + v_2.$$
By compactness of $v$ we can find $v_2' \ll v_2$ such that $v \le v_1+v_2'$. Now,
$$v_2' \ll v_2 \le u+v \le Nx'' \le Nu + Nv_1',$$
and so we can use (P1) (leaving $Nu$ unchanged) to find $x_1, \dots, x_N$ such that 
$$v_2' \le Nu + \sum_{j=1}^Nx_j, \qquad x_j \le v_1', \qquad x_j \le v_2.$$
It follows that $2x_j \le v_1'+v_2 \le v$ and that
$$2v \le 2v_1 + 2v_2' \le 2x + 2Nu + 2 \sum_{j=1}^Nx_j \le v + (2N+1)u +2 \sum_{j=1}^Nx_j,$$
as desired.
\end{proof}

\noindent The corollary below illustrates how the preceding lemma can be used to find elements with bad divisibility properties:

\begin{corollary} \label{cor:2-div-obstruct}
Let $X$ be a compact  Hausdorff space and suppose that $p\in C(X)\otimes \mathcal K$ is a projection such that $[1]\nleq (2N+1)[p]$ in $K_0(C(X))$, where $1$ denotes the unit of $C(X)$. Then $\Dv_2(\langle 1\rangle+\langle p\rangle, C(X))> N$.  
\end{corollary}

\begin{proof}
Consider the compact elements $u=\langle p\rangle$ and $v=\langle 1\rangle $ of $\Cu(C(X))$. Note that $2x \le \langle 1 \rangle$ implies $x=0$ and that $2 \langle 1 \rangle \nleq \langle 1 \rangle + (2N+1) \langle p \rangle$.  The desired conclusion now follows from Lemma~\ref{lm:2-div}. 
\end{proof}

\noindent The relation $[1]\nleq (2N+1)[p]$ in $K_0(C(X))$ is satisfied whenever the $(2N+1)$-fold direct sum of $p$ with itself is a projection with non-trivial Euler class (as explained in more detail below). It is known that for each integer $d \ge 1$ and for each positive integer $N$ there exist $X$ and $p \in C(X) \otimes \mathcal K$ such that $(2N+1)[p]$ has non-trivial Euler class and $p$ has rank $d$.  The unital \Cs{} $A=(p \oplus 1) \big( C(X) \otimes \cK\big)(p \oplus 1)$ with this choice of $X$ and $p$ will then satisfy $\rnk(A) = d$ and $\Dv_2(A) > N$. 

We will now give a different method for constructing elements with large divisibility constants and large rank, where we get upper bounds and where we also can give sharper lower bounds for the weak divisibility constant. 
Let $S^2$ denote the 2-dimensional sphere.
Let $p$ denote the ``Bott-projection" in $C(S^2) \otimes M_2 \subseteq C(S^2) \otimes \cK$, i.e., the projection associated to the Hopf line bundle over $S^2$. 
For each $1 \le j \le N$, let $p_j \in C((S^2)^N) \otimes \cK$ be given by
\[
p_j(x_1,x_2, \dots, x_N) = p(x_j), \qquad (x_1, \dots, x_N) \in (S^2)^N.
\]
Since $\langle 1\rangle \leq 2\langle p\rangle $ in $\Cu(C(S^2))$, we have
\[
N\langle 1\rangle\leq 2\Big\langle \bigoplus_{i=1}^Np_i\Big\rangle.
\]

\noindent
As another obstruction to weak divisibility, we shall use the following corollary of Lemma~\ref{lm:Z} (iii), cf.\ the remarks at the beginning of this section,  applied to the relations $\langle 1 \rangle \le 2 \langle p_i \rangle$, $i=1,2 \dots, N$.

\begin{corollary} \label{cor:key}
Let $X$ be a locally compact Hausdorff space, and let $q\in C_0(X)\otimes\mathcal K$ be a projection. Let $(p_i)_{i=1}^N$ be the projections in $C((S^2)^N)\otimes\mathcal K$ defined in the preceding paragraph. Suppose that
\[
 q\otimes 1 \; \nprecsim  \; q\otimes \bigoplus_{i=1}^Np_i.
\]
Then $\dv_2(\langle q \rangle, C_0(X))>N$.
\end{corollary}

\noindent
Let us now give examples of projections  to which the corollary above can be applied. 
We will make use of characteristic classes of vector bundles. 
Recall that projections in  $C(X) \otimes \cK$, with $X$ compact and Hausdorff, give rise to vector bundles over $X$: if $p$ is a projection, then
$\eta_p=(E_p,X,\pi)$, with $E_p=\{(x,v)\in X\times l_2(\N)\mid p(x)v=v\}$ is the vector bundle associated to $p$.
Up to Murray-von Neumann equivalence of projections and isomorphism of vector bundles, this correspondence is a bijection.
We denote by $e(\eta_p)\in H^*(X)$, or simply $e(p)$, the Euler class of $\eta_p$. 
For the cartesian product of spheres  $(S^2)^N$ we have (e.g., by the K\"unneth formula)
that 
\[
H^*((S^2)^N)\cong \C[z_1,z_2,\dots,z_N]/(z_1^2,z_2^2,\dots,z_N^2).\] 
With this identification, the Euler classes  of the projections $p_i\in C((S^2)^N)\otimes \mathcal K$ defined earlier can be shown to be $e(p_i)=z_i$.

\begin{proposition} \label{prop:euler-obstruct}
Let $X$ be a compact Hausdorff space and let $q\in C(X)\otimes \cK$ be a projection such that $e(q)^N\neq 0$. 
Then
\[
\dv_2(\langle 1\oplus q\rangle,C(X))>N.
\] 
\end{proposition}
\begin{proof}
By Corollary \ref{cor:key} it suffices to show that 
\begin{align}\label{thetrick}
(1_X\oplus q)\otimes 1_{(S^2)^N}  \nprecsim (1_X \oplus q)\otimes \bigoplus_{i=1}^N p_i
\end{align}
in $\Cu(C(X) \otimes C((S^2)^N))$, where $1_X$ denotes the unit in $C(X)$. (In the formulation of the proposition above, we denoted $1_X$ simply by $1$.)
Observe that the trivial rank 1 projection is a subprojection of the projection on the left-hand side of \eqref{thetrick}. Thus, it suffices to show that
the right side of \eqref{thetrick} has non-zero Euler class.

Set $\rnk(q)=k$. For each positive integer $i$, let $c_i(q)\in H^{2i}(X)$ denote the $i$th characteristic class of $q$
(so that $c_{k}(q)=e(q)$). By the K\"unneth Theorem (\cite[Theorem A.6]{milnor1974characteristic}), we can identify 
$H^{*}(X\times (S^2)^N)$ with  $H^{*}(X)\otimes H^{*}((S^2)^N)$. Then
\begin{align*}
e\big((1_X\oplus q)\otimes \bigoplus_{i=1}^N p_i\big) 
&=e\big(1_X \otimes \bigoplus_{i=1}^N p_i\big) \, e\big(q\otimes \bigoplus_{i=1}^N p_i \big)\\
&=\prod_{i=1}^N e(1_X\otimes p_i) \, \prod_{i=1}^N e(q\otimes p_i)\\
&=\prod_{i=1}^N e(p_i) \, \prod_{i=1}^N\sum_{j=0}^k c_{k-j}(q)e(p_i)^j\\
&=\prod_{i=1}^N e(p_i) \, e(q)^N\neq 0.
\end{align*}
In the above computation we have used that $e(q\otimes p)=\sum_{j=0}^k c_{k-j}(q)e(p)^{j}$, for $q$ a projection of rank $k$ and $p$ a projection of rank 1. To obtain  the last equality we have used that $e(p_i)^2=0$ for all $i$.   
\end{proof}

\noindent
Let us now give examples of families of projections  to which the above proposition can be applied. We shall here and in the following, whenever $p$ is a projection (in a \Cs) and $n$ is a positive integer, let $n \! \cdot \! p$ denote the $n$-fold direct sum, $p \oplus p \oplus \cdots \oplus p$, (in a matrix algebra over the given \Cs) of the projection $p$.  

\begin{example}
Let $N$ be a positive integer, and let $\C\mathrm{P}^N$ denote the $2N$-dimensional complex projective space. Let $\eta$ denote the tautological line bundle over $\C\mathrm{P}^N$
and $p_{\eta}$ the rank 1 projection associated to it. It is known that $e(p_\eta)=z^2\in C[z^2]/(z^{2N})$, where we have
identified $H^*(\C\mathrm{P}^N)$ with $C[z^2]/(z^{2N})$. Let $d,d'$ be positive integers such that $dd'<N$. Then $e(d \!\cdot \! p_\eta)^{d'}=z^{2dd'}\neq 0$.
It follows that
\[
\left\lfloor\frac{N-1}{d}\right\rfloor <  
\dv_2(\langle 1\oplus d \!\cdot \!  p_\eta \rangle ,C(\C\mathrm{P}^N))
\le \Dv_2(\langle 1\oplus d \!\cdot \!  p_\eta \rangle,C(\C\mathrm{P}^N)) \le \left\lceil \frac{N+d+1/2}{\lfloor d/2 \rfloor}\right\rceil .
\]
Indeed, the first inequality follows from Proposition~\ref{prop:euler-obstruct} and the calculations made above. The second inequality follows from Proposition~\ref{prop:a}. The last inequality can be proved as follows: Put $x = \lfloor d/2 \,\rfloor \langle p_\eta \rangle$. Then $2x \le \langle 1\oplus d \!\cdot \!  p_\eta \rangle$. By a classical result about vector bundles (see \cite[Chapter 9, Proposition 1.1]{Hus:fibre}) we have that $1 \precsim k \!\cdot \!  p_\eta$ if $2N \le 2k-1$. It follows that
$$1\oplus d \!\cdot \!  p_\eta \precsim n \, \lfloor d/2 \,\rfloor \, p_\eta,$$
or, equivalently, that $\langle 1\oplus d \!\cdot \!  p_\eta \rangle \le nx$, if $n \, \lfloor d/2 \,\rfloor \ge N+d+1/2$. 
\end{example}

\subsection*{Simple \Cs s with bad divisibility}

\noindent In this and the following two subsections we give examples of unital simple \Cs s with bad divisibility behaviour. They use the Euler class obstruction described in the following example. 

\begin{example} \label{ex:spheres}
Let $d$ be a positive integer. 
Following the notation in \cite{Ror:simple}, for each set $I = \{i_1,i_2, \dots, i_k\} \subseteq \{1,2, \dots, d\}$, let $p_I$ be the one-dimensional projection in $C((S^2)^d) \otimes \cK$ given by
\[
p_I(x) = p_{i_1}(x) \otimes p_{i_2}(x) \otimes \cdots \otimes p_{i_k}(x), \qquad x \in (S^2)^d,
\]
where $p_i$ is as defined above Corollary~\ref{cor:key}. It is  shown in \cite[Proposition 4.5]{Ror:simple} that if $I_1, I_2 \dots, I_r$ are subsets of $\{1,2, \dots, d\}$ that admit a matching (i.e., $|\bigcup_{i \in F} I_i| \ge |F|$ for all subset $F$ of $\{1,2, \dots, r\}$) then the Euler class of $p_{I_1} \oplus p_{I_2} \oplus \cdots \oplus p_{I_r}$ is non-zero. 
\end{example}

\noindent  The examples constructed in this and the following two subsectons are built on the same template described in the following lemma (which is a variation of one of Villadsen's constructions). We retain the terminology from the example above throughout the rest of this section. 

\begin{lemma} \label{lm:template} Let $\big(J_j\big)_{j=1}^\infty$ be a sequence of  pairwise disjoint subsets of $\N$. Choose $d_n$ large enough so that all $J_j$,  $j=1,2, \dots, n$, are contained in the set  $\{1,2, \dots, d_n\}$. (This, of course, can be accomplished by taking $d_n = \sum_{j=1}^n |J_j|$.) Consider the projection $q_n$ of rank $2^n$ in $C((S^2)^{d_n}) \otimes \cK$ given by
$$q_n = 1 \oplus p_{J_1} \oplus 2 \! \cdot \! p_{J_2} \oplus \cdots \oplus 2^{n-1} \! \cdot \! p_{J_n}.$$ 
It follows that there is a simple unital AH-algebra $A$ which is the inductive limit of the sequence
$$\xymatrix{ q_1 \big( C((S^2)^{d_1}) \otimes \cK\big) q_1 \ar[r]^-{\varphi_1} & q_2 \big( C((S^2)^{d_2}) \otimes \cK\big) q_2 \ar[r]^-{\varphi_2} &  \cdots \ar[r] & A,}$$
where the connecting mappings $\varphi_n$ are unital. 
\end{lemma} 

\begin{proof} Set $X_n = (S^2)^{d_n}$ and $A_n = q_n \big( C(X_n) \otimes \cK\big) q_n$.  Write $$X_{n+1} = X_n \times (S^2)^{d_{n+1}-d_n},$$  let $\pi_n \colon X_{n+1} \to X_n$ be the projection mapping, and let $\pi_{m,n} \colon X_n \to X_m$ denote the composition map $\pi_m \circ \pi_{m+1} \circ \cdots \circ \pi_{n-1}$. Choose $x_n \in X_n$ for each $n$ such that the set $ \{\pi_{m,n}(x_n) \mid n \ge m\}$ is dense in $X_m$ for all $m \ge 1$. 

Define a \sh{} $\varphi_n^0 \colon C(X_n,\cK) \to C(X_{n+1}, \cK)$ by
$$\varphi_n^0(f)(x) = f(\pi_n(x)) \, \oplus \,  \big(f(x_n) \otimes p_{J_{n+1}}(x)\big), \qquad f \in C(X_n,\cK), \quad x \in X_{n+1},$$
where we in a suitable way have identified $\cK \oplus (\cK \otimes \cK)$ with a subalgebra of $\cK$. We also identify $C(X) \otimes \cK$ with $C(X, \cK)$. We make another identification: if 
$$J \subseteq \{1,2, \dots, d_n\} \subseteq \{1,2, \dots, d_{n+1}\} ,$$
then the projection $p_J$ is defined both in $C(X_n) \otimes \cK$ and in $C(X_{n+1}) \otimes \cK$ and $p_{J_n} = p_{J_n} \circ \pi_n$ (where the former occurrence of $p_{J_n}$ is viewed as an element in the former algebra, and the latter in the latter). We shall use the same notation for the two projections. Taking these identification a step further, we have $q_n = q_n \circ \pi_n$ and that $q_{n+1} = q_n \oplus 2^n \! \cdot \! p_{J_{n+1}}$. (These identification hold, strictly speaking, only up to conjugation with an inner automorphism on $\cK$.) In this notation we get
\begin{eqnarray*}
\varphi_n^0(q_n)(x) & = & q_n(\pi_n(x)) \, \oplus \,  \big(q_n(x_n) \otimes p_{J_{n+1}}(x) \big) \\ 
&= &q_n(\pi_n(x)) \, \oplus \, \rnk(q_n) \! \cdot \! p_{J_{n+1}}(x) \; = \; q_{n+1}(x),
\end{eqnarray*}
for all $x \in X_{n+1}$, i.e., $\varphi_n^0(q_n) = q_{n+1}$ (possibly after composing $\varphi_n^0$ with an inner automorphism on $\cK$). This shows that $\varphi_n^0$ maps $A_n$ unitally into $A_{n+1}$. Let $\varphi_n$ denote the unital \sh{} that arises in this way, i.e., $\varphi_n$ is the restriction of $\varphi_n^0$ to $A_n$ (and the co-restriction to $A_{n+1}$), and let $A$ be the inductive limit of the sequence
$$\xymatrix{A_1 \ar[r]^-{\varphi_1} & A_2 \ar[r]^-{\varphi_2} & A_3 \ar[r]^-{\varphi_3} & \cdots \ar[r] & A.}$$
Let $\varphi_{m,n} \colon A_m \to A_n$ denote the composition map $\varphi_{n-1} \circ \varphi_{n-2} \circ \cdots \circ \varphi_{m}$ when $m \le n$. One can check that $\varphi_{m,n}(f)(x)$ is non-zero for all $x \in X_n$ if $f$ is a function in $A_m = q_m C(X_m, \cK)q_m$ which is non-zero on at least one point in the set $\{\pi_{m,k}(x_k) \mid m \le k \le n\}$. By the choice of the points $x_n$, it follows that for each $m$ and for each non-zero $f$ in $A_m$ there is $n \ge m$ such that $\varphi_{m,n}(f)$ is full in $A_n$ (i.e., that $\varphi_{m,n}(f)(x) \ne 0$ for all $x \in X_n$). This entails that $A$ is simple.
\end{proof}

\begin{lemma} \label{lm:simple-1}
Let $N$ be a positive integer. In the notation of Lemma~\ref{lm:template} choose the sequence $(J_j)_{j=1}^\infty$ such that $|J_j| =  N \! \cdot \! 2^{n-1}$. It then follows that 
$$\dv_2\big(\langle q_n \rangle, C((S^2)^{d_n}) \big) > N, \qquad \Dv_2\big(\langle q_2 \rangle, C((S^2)^{d_2}\big) \le 3N+4.$$
for all $n$.
\end{lemma}

\begin{proof} We use Proposition~\ref{prop:euler-obstruct} to prove the first claim. It suffices to show that the Euler class of the projection $N \! \cdot \! p_{J_1} \oplus 2N \! \cdot \! p_{J_2} \oplus \cdots \oplus 2^{n-1}N \! \cdot \! p_{J_n}$ is non-zero. But this follows from \cite[Proposition 4.5]{Ror:simple}, cf.\ Example~\ref{ex:spheres} above, and from the choice of the sets $J_n$. 

To prove the second claim, put $x = \langle p_{J_2} \rangle$ and note that $2x \le \langle q_2 \rangle$. It follows from \cite[Proposition 1]{Dupre:vectorbundle} that $q_2 \precsim M \! \cdot \! p_{J_2}$ if $M-4 \ge (2d_2-1)/2 = 3N-1/2$. This shows that $\langle q_2 \rangle \le (3N+4)x$.
\end{proof}

\begin{theorem} \label{thm:simple}
For each positive integer $N$ there exists a simple unital infinite dimensional AH-algebra $A$ such that $N < \dv_2(A) \le \Dv_2(A) \le 3N+4$.
\end{theorem}

\begin{proof} Let $A$ be the simple \Cs{} constructed in Lemma~\ref{lm:template} based on the choice of $(J_j)_{j=1}^\infty$ made in Lemma~\ref{lm:simple-1}. Then $A$ is the inductive limit of the sequence of \Cs s $A_n = q_n \big( C((S^2)^{d_n}) \otimes \cK\big) q_n$ with unital connecting mappings. It follows from Lemma~\ref{lm:simple-1} that $\dv_2(A_n) > N$ for all $n$, and that $\Dv_2(A_2) \le 3N+4$.

By Proposition~\ref{prop:inductive-limits} and Remark~\ref{rem:image}, 
$$\dv_2(A) = \inf_{n \in \N} \dv_2(A_n) > N,$$
and $\Dv_2(A) \le \Dv_2(A_2) \le 3N+4$.
\end{proof}

\begin{remark}[Initial objects] Suppose that $\cC$ is a class of unital \Cs. An element $A$ in $\cC$ is an \emph{inital object in $\cC$} if there exists a unital \sh{} $A \to B$ for every $B$ in $\cC$. 

It is well-known that the Cuntz algebra $\cO_\infty$ is an initial object in the class of unital properly infinite \Cs s. In fact, a unital \Cs{} is properly infinite if and only if it contains $\cO_\infty$ as a unital sub-\Cs. Every properly infinite unital sub-\Cs{} of $\cO_\infty$ is then also an initial object in the class of unital properly infinite \Cs s. Hence the Cuntz-Toeplitz algebras, $\cT_n$, $n \ge 2$, are initial objects and so are all unital Kirchberg algebras $A$ for which the assignment $[1_A] \mapsto 1$ extends to a homomorphism $K_0(A) \to \Z$.

It was shown in \cite{EllRor:Hausdorff} that also the class of unital \Cs s of real rank zero and of infinite rank has initial objects. One can even find initial objects to this class which are simple AF-algebras (necessarily with infinite dimensional trace simplex). It follows in particular that the class of unital simple infinite dimensional \Cs s of real rank zero has initial objects. 

Clearly, $\C$ is an initial object in the category of all unital \Cs s, and so is any unital \Cs{} that admits a character. (Note that we do not require the unital \sh{} $A \to B$ to be injective.)

The corollary below shows that initial objects do not exist in the general non-real rank zero case.
\end{remark}

\begin{corollary} The class of unital simple infinite dimensional \Cs s and the class of unital \Cs s of infinite rank do not have initial objects. In fact, there is no unital \Cs{} without characters that maps unitally into every unital simple infinite dimensional \Cs. 
\end{corollary}

\begin{proof} If $A$ is a unital  \Cs{} that maps unitally into every unital simple infinite dimensional \Cs, then $\dv_2(A) \ge \dv_2(B)$ for all unital simple infinite dimensional \Cs s $B$, cf.\ Remark~\ref{rem:image}, whence $\dv_2(A) = \infty$ by Theorem~\ref{thm:simple}. On the other hand, if $A$ has no character, then $\dv_2(A) < \infty$ by Corollary~\ref{cor:char}.
\end{proof}

\subsection*{The asymptotic divisibility numbers}

\noindent We can give a lower and an upper bound on the asymptotic divisibility constant (discussed in Section~\ref{sec:stable-div}) for the \Cs{} considered above:

\begin{corollary} \label{cor:simple-asymp} Let $N$ be a positive integer, and let $A$ be the simple AH-algebra constructed in Theorem~\ref{thm:simple} associated with $N$. It follows that
$$(N-1)/2 < \Dv_*(A) \le 2N+2.$$
\end{corollary}

\begin{proof} By Proposition~\ref{prop:almdiv} we get that $\Dv_*(A) \ge (\Dv_2(A)-1)/2 > (N-1)/2$. 
To prove the reserve inequality, take any positive integer $n$ and put $m=2^{n-1}$. We show that $\Dv_m(A_n) \le (2N+2)m$, where $A_n$ is as in the proof of Theorem~\ref{thm:simple}. 
In the notation of Lemma~\ref{lm:simple-1}, let $x = \langle p_{J_n} \rangle$, put $u = \langle q_n \rangle$, and recall that $q_n$ is the unit of the \Cs{} $A_n$. By the definition of $q_n$ (in Lemma~\ref{lm:simple-1}) it follows that $mx \le u$. As
$$\dim((2+2N)m \! \cdot \! p_{J_n}) - \dim(q_n) = (2N+2)m -2^n = 2^n N \ge \frac{\dim(X_n)-1}{2},$$
if follows from \cite[Proposition 1]{Dupre:vectorbundle} that $(2+2N)m \! \cdot \! p_{J_n} \precsim q_n$, whence $u \le (2N+2)mx$. This proves that $\Dv_m(A_n) \le (2N+2)m$.

It follows from Remark~\ref{rem:image} that $\Dv_m(A) \le (2N+2)m$ whenever $m$ is a power of 2, and this entails that
$$\Dv_*(A) = \liminf_{m \to \infty} \Dv_m(A)/m \le \liminf_{n \to \infty} \Dv_{2^{n-1}}(A)/2^{n-1} \le 2N+2,$$
as desired.
\end{proof}

\noindent We can use Lemma~\ref{lm:template} to construct a  unital, simple AH-algebra $A$ such that $\Dv_*(A) = \infty$. The proof requires the following sharpening of Corollary~\ref{cor:key} that may have independent interest.

\begin{corollary} \label{cor:key-m}
Let $X$ be a locally compact Hausdorff space, and let $q\in C_0(X)\otimes\mathcal K$ be a projection. Let $m$ and $N$ be positive integers, let $I_1, I_2, \dots, I_N$ be pairwise disjoint subsets of $\N$ with $|I_i| = m-1$ for all $i$, and let $(p_{I_i})_{i=1}^N$ be the associated projections in $C((S^2)^{(m-1)N})\otimes\mathcal K$ defined in Example~\ref{ex:spheres}. Suppose that
\[
 q\otimes 1 \; \nprecsim  \; q\otimes \bigoplus_{i=1}^Np_{I_i}.
\]
Then $\dv_m(\langle q \rangle, C_0(X))>N$.
\end{corollary}

\begin{proof} Apply Lemma~\ref{lm:Z} (iii) to $x = \langle 1 \rangle$ and $y_i = \langle p_{I_i} \rangle$, and note that $x \le my_i$, cf.\ Example~\ref{ex:spheres}.
\end{proof}

\begin{lemma} \label{lm:simple-2}
Let $(J_j)_{j=1}^\infty$ be a sequence of  pairwise disjoint subsets of $\N$ with $|J_j| = 2^{2j-1} \, j$. Then, in the notation of Lemma~\ref{lm:template}, we have 
$$2^k \, k< \dv_{2^k}\big(\langle q_n \rangle, C((S^2)^{d_n}) \big) < \infty$$
if $n \ge k$, and that $\dv_{2^k}\big(\langle q_n \rangle, C((S^2)^{d_n}) \big) = \infty$ if $n < k$.
\end{lemma}

\begin{proof} We use Corollary~\ref{cor:key-m} with $N=2^k \, k$ and $m= 2^k$ to prove the first claim. As $1 \precsim q_n$ it suffices to show that $q_n \otimes \bigoplus_{i=1}^Np_{I_i}$ has non-trivial Euler class, when $I_1, \dots, I_N$ are as in  Corollary~\ref{cor:key-m}. Write
$$q_n \otimes \bigoplus_{i=1}^Np_{I_i} \;  = \; \bigoplus_{j=0}^n \, \bigoplus_{i=1}^{N} 2^{\max\{0,j-1\}} \! \cdot \! p_{J_j} \otimes  p_ {I_i} = \bigoplus_{j=0}^n \, \bigoplus_{i=1}^{N} 2^{\max\{0,j-1\}} \! \cdot \! p_{J_j \cup I_i}.$$
As explained in Example~\ref{ex:spheres}, to prove non-triviality of the Euler class of the projection $q_n \otimes \bigoplus_{i=1}^Np_{I_i}$  one needs to verify the combinatorial fact that the family of sets $(J_j \cup I_i)$, $j=0, \dots, n$, $i=1, \dots, N$, and with the set $J_j \cup I_i$ repeated $2^{\max\{0,j-1\}}$ times, satisfies the Marriage Lemma condition. 

By first exhausting the elements in the sets $I_i$, and using that $\sum_{j=0}^{k-1} 2^{\max\{0,j-1\}} = 2^{k-1} < |I_i|$, it suffices to show that the family of sets  $(J_j)$, $j=k, \dots, n$, with each set  repeated $2^{j-1}N=2^{j+k-1}k$ times, satisfies the Marriage Lemma condition. However, this holds because $|J_j| = 2^{2j-1}j \ge 2^{j+k-1}k$ when $j \ge k$.

The second claim follows from the fact that the dimension of the projection $q_n$ is $2^n$ and that $\dv_m(\langle q_n \rangle, C((S^2)^{d_n})) = \infty$ whenever $m > \dim(q_n)$.
\end{proof}

\begin{theorem} \label{thm:simple-2} There is a simple unital infinite dimensional AH-algebra $A$ which satisfies $\Dv_*(A) = \infty$.
\end{theorem}

\begin{proof} Let $A$ be the simple AH-algebra contructed in Lemma~\ref{lm:template} with respect to the choice of $(J_j)$ from Lemma~\ref{lm:simple-2}. Recall that $A$ is an inductive limit of a sequence of unital \Cs s $A_1 \to A_2 \to \cdots$, where
$$A_n = q_n \big( C((S^2)^{d_n}) \otimes \cK\big) q_n.$$
It therefore follows from Lemma~\ref{lm:simple-2} that $\dv_{2^k}(A_n) > 2^k \, k$, when $n\ge k$, and $\dv_{2^k}(A_n) = \infty$ when $n < k$. This entails that $\Dv_{2^k}(A) > 2^k \, k$ by Proposition~\ref{prop:inductive-limits}.  Finally, by Proposition~\ref{prop:almdiv} (ii), 
\[
\Dv_*(A) = \limsup_{k \to \infty} \Dv_{2^k}(A)/2^k  = \infty.\qedhere
\]
\end{proof}

\noindent As remarked in Section~\ref{sec:stable-div}, if $A$ is any unital \Cs, then $\Dv_*(A) = 0$ if and only if $A$ is properly infinite, and $\Dv_*(A) \in [1,\infty]$ otherwise. Moreover, $\Dv_*(A) = 1$ if and only if $A$ is almost divisible (and not properly infinite).  In other words, the range of the invariant $\Dv_*( \, \cdot \,)$ is contained in the set $\{0\} \cup [1,\infty]$, and $\Dv_*(A) \le 1$ if and only if $A$ is almost divisible. 

We can easily produce examples of simple, unital, infinite dimensional \Cs s $A$ such that $\Dv_*(A) = 0$ (eg., $A$ could be a Cuntz algebra), or such that $\Dv_*(A) = 1$ (eg., $A$ is any simple, unital, infinite-dimensional \Cs{} of real rank zero,  cf.\ Example~\ref{ex:rr0}). The theorem above provides an example of a  simple, unital, infinite dimensional \Cs{} $A$ where $\Dv_*(A) = \infty$. 

It follows from Corollary~\ref{cor:simple-asymp} that  $\Dv_*(\, \cdot \,)$ attains infinitely many values in the interval $(1,\infty)$, when restricted to the class of unital, simple, infinite dimensional \Cs s, and that the possible values of  $\Dv_*(\, \cdot \,)$  in this interval is upwards unbounded. We do not know if all values in the interval $(1,\infty)$ are thus attained. For that matter we cannot exhibit any number in the interval $(1,\infty)$ which for sure is the value of $\Dv_*(A)$ for some simple unital infinite dimensional \Cs{} $A$. 

\subsection*{Divisibility of infinite tensor products}
\noindent
We end this section by giving yet another class of examples of simple unital \Cs s with bad divisibility properties. The ones we construct below are of the form $\bigotimes_{j=1}^\infty A_j$, where the $A_j$'s are unital simple infinite dimensional \Cs s. In particular, such \Cs s need not absorb the Jiang-Su algebra tensorially. It remains an open problem if $\bigotimes_{j=1}^\infty A$ absorbs the Jiang-Su algebra whenever $A$ is a simple unital infinite dimensional \Cs{} (or a unital \Cs{} without characters), cf.\ \cite{DadToms:Z}.

It was shown in \cite[Example 4.8]{HirRorWin:C(X)} that there exists a sequence $(A_n)$ of homogeneous \Cs s of rank two such that $\bigotimes_{n=1}^\infty A_n$ does not absorb the Jiang-Su algebra tensorially. (It is an easy conseqence of this that the Jiang-Su algebra cannot embed unitally into $\bigotimes_{n=k}^\infty A_n$ for some $k$.) Of course, one can regroup the tensor factors $A_n$ to get a new sequence $(B_n)$ of unital \Cs s each of which has infinite rank and where the Jiang-Su algebra does not embed into $\bigotimes_{n=1}^\infty B_n$. It is not known if every unital \Cs{} of infinite rank admits an embedding of a unital simple infinite dimensional \Cs. If it were true, then Theorem~\ref{thm:inf-tensor} below would follow from \cite[Example 4.8]{HirRorWin:C(X)} .

We introduce some notation to keep track of the combinatorics. Define a total order on the set $\N \times \N_0$ by 
$$(k,j) \le (\ell,i) \iff k+j < \ell + i \quad \text{or} \quad  (k+j=\ell+i \;  \text{and} \;  k < \ell).$$ 
For each $(k,j) \in \N \times \N_0$ and for each integer $m \ge k$ let $S(m;k,j)$ denote the set of all $m$-tuples $(i_1,i_2, \dots, i_m) \in \N_0^m$ such that $i_k = j$ and $(\ell,i_\ell) < (k,j)$ for all $\ell \ne k$.

\begin{lemma}  \label{lm:inf-tensor}
Let $N \ge 1$ be an integer. For each integer $k \ge 1$, let $(J_j^{(k)})_{j=1}^\infty$ be a sequence of subsets of $\N$ such that $J_j^{(k)} \cap J_i^{(\ell)} = \emptyset$ whenever $(k,j) \ne (\ell,i)$, and such that
$$|J_j^{(k)}| \; = \; \max_{m \ge k} \, \sum_{(i_1,\dots, i_m) \in S(m;k,j)} \, N \, \prod_{t=1}^m 2^{\max\{i_t-1,0\}}.
$$
(The quantity on the right-hand side is finite because $S(m;k,j)$ is finite for all $(k,j)$ and all $m \ge k$, and $S(m;k,j) = \emptyset$ when $m > k+j$.) 
Let $d_n^{(k)} \in \N$ and $q_n^{(k)} \in C((S^2)^{d_n^{(k)}}) \otimes \cK$ be as defined in Lemma~\ref{lm:template} associated with the sequence $(J_j^{(k)})_{j=1}^\infty$. It then follows that
$$\dv_2\Big(\big\langle q_n^{(1)} \otimes q_n^{(2)} \otimes \cdots \otimes q_n^{(m)} \big\rangle, C((S^2)^{d_n^{(1)}}) \otimes C((S^2)^{d_n^{(2)}}) \otimes \cdots \otimes C((S^2)^{d_n^{(m)}})\Big) > N,$$
for all positive integers $n$ and $m$.
\end{lemma}

\begin{proof}  Let $T(n,m)$ be the set of all non-zero $m$-tuples  $(i_1,i_2, \dots, i_m) \in \N_0^m$ such that $i_k \le n$ for all $k=1,2, \dots, m$. Adopt the convention $J_0^{(k)} = \emptyset$ for all $k$ and let $p_\emptyset$ denote the trivial (= constant) one-dimensional projection. We can then express the projection $q_n^{(1)} \otimes q_n^{(2)} \otimes \cdots \otimes q_n^{(m)} $ as follows:
$$1 \oplus \sum_{(i_1,\dots, i_m) \in T(n,m)} \big( \prod_{t=1}^m 2^{\max\{i_t-1,0\}}\big) \cdot  p_{J_{i_1}^{(1)}} \otimes p_{J_{i_2}^{(2)}} \otimes \cdots \otimes p_{J_{i_m}^{(m)}}.$$
By Proposition~\ref{prop:euler-obstruct} it suffices to show that the Euler class of the projection 
$$\sum_{(i_1,\dots, i_m) \in T(n,m)} \big(N \, \prod_{t=1}^m 2^{\max\{i_t-1,0\}}\big) \cdot  p_{J_{i_1}^{(1)}} \otimes p_{J_{i_2}^{(2)}} \otimes \cdots \otimes p_{J_{i_m}^{(m)}}$$ 
is non-zero, or, equivalently, that the Euler class of the projection
$$\sum_{(i_1,\dots, i_m) \in T(n,m)} \big(N \, \prod_{t=1}^m 2^{\max\{i_t-1,0\}}\big) \cdot  p_{{J_{i_1}^{(1)}} \cup {J_{i_2}^{(2)}} \cup \cdots \cup {J_{i_m}^{(m)}}}$$ 
is non-zero. By  \cite[Proposition 4.5]{Ror:simple}, cf.\ Example~\ref{ex:spheres}, it suffices to show that the family of sets ${J_{i_1}^{(1)}} \cup {J_{i_2}^{(2)}} \cup \cdots \cup {J_{i_m}^{(m)}}$, where $(i_1, \dots, i_m) \in T(n,m)$ and where the set ${J_{i_1}^{(1)}} \cup {J_{i_2}^{(2)}} \cup \cdots \cup {J_{i_m}^{(m)}}$ is repeated $N \cdot \prod_{t=1}^m 2^{\max\{i_t-1,0\}}$ times, admits a matching. 

Construct a matching by selecting the matching elements for the $N \cdot \prod_{t=1}^m 2^{\max\{i_t-1,0\}}$ copies of the set ${J_{i_1}^{(1)}} \cup {J_{i_2}^{(2)}} \cup \cdots \cup {J_{i_m}^{(m)}}$ inside the subset $J_{j}^{(k)}$, where $(k,j)$ is the largest of the elements $(1,i_1), (2,i_2), \dots, (m,i_m)$. To check that this work, i.e., to see that $J_j^{(k)}$ is large enough, let $T(n,m;k,j)$ be the set of those $m$-tuples $(i_1,i_2, \dots, i_m)$ in $T(n,m)$ for which 
$$(k,j) = \max\{(1,i_1), (2,i_2), \dots, (m,i_m)\}.$$ 
Then $T(n,m;k,j) \subseteq S(m;k,j)$, and so it follows by the assumption on $|J_j^{(k)}|$ that
$$|J_j^{(k)}| \ge \sum_{(i_1, \dots, i_m) \in T(n,m;k,j)} \, N \, \prod_{t=1}^m 2^{\max\{i_t-1,0\}}.$$
The suggested matching is therefore possible. 
\end{proof}

\noindent The theorem below shows that an infinite tensor product of simple unital infinite dimensional \Cs s does not necessarily have good divisibility properties. Any such \Cs{} $A = \bigotimes_{n=1}^\infty A_n$ will have (many) non-trivial central sequences, i.e., the central sequence algebra $A_\omega \cap A'$ with respect to an ultrafilter $\omega$ on $\N$ is non-trivial. For example, $CM_m(\C)$ embeds into $A_\omega \cap A'$ for all $m$, albeit, not necessarily with full image. However, in the example contructed below, one cannot embed the Jiang-Su algebra into $A_\omega \cap A'$.

\begin{theorem} \label{thm:inf-tensor} For each integer $N > 2$, there exist a sequence $(A_n)$ of unital simple infinite dimensional \Cs s (in fact, AH-algebras) such that
$$\dv_2\Big( \bigotimes_{n=1}^\infty A_n \Big) > N. $$
In particular, $\bigotimes_{n=1}^\infty A_n$ does not absorb the Jiang-Su algebra $\cZ$, and it does not even admit a unital embedding of $\cZ$.
\end{theorem}

\begin{proof} Let $A_k$ be the simple unital AH-algebra constructed in Lemma~\ref{lm:template} associated with the sequence $\big(J_j^{(k)}\big)_{j=1}^\infty$ from Lemma~\ref{lm:inf-tensor}. Then $A_k$ is an inductive limit of a sequence, $A_k(1) \to A_k(2) \to \cdots$, of unital homogenous \Cs s with unital connecting maps, where
$$A_k(n) = q_n^{(k)} \Big( C((S^2)^{d_n^{(k)}}) \otimes \cK \Big) q_n^{(k)}.$$

The infinite tensor product $ \bigotimes_{n=1}^\infty A_n$ is the inductive limit of the sequence
$$A_1 \to A_1 \otimes A_2 \to A_1 \otimes A_2 \otimes A_3 \to \cdots$$
with unital connecting maps. It therefore suffices to show that
$\dv_2\big( \bigotimes_{k=1}^m A_k\big) > N$
for every $m$, cf.\ Proposition~\ref{prop:inductive-limits}. Now, $\bigotimes_{k=1}^m A_k$ is the inductive limit of the sequence 
$$\bigotimes_{k=1}^m A_k(1) \to \bigotimes_{k=1}^m A_k(2) \to \bigotimes_{k=1}^m A_k(3)  \to \cdots,$$
with unital connecting mappings, and so, again by Proposition~\ref{prop:inductive-limits}, it suffices to show that
$$\dv_2\Big( \bigotimes_{k=1}^m A_k(n)\Big) > N$$
for every $m$ and $n$. The latter is precisely the content of Lemma~\ref{lm:inf-tensor}. 
\end{proof}

\section{Ultrapowers} \label{sec:ultrapowers}

\noindent In this section we show that our divisibility properties behave well with respect to taking direct products and ultrapowers of sequences of unital \Cs s. This has the surprising consequence that such products and ultrapowers may admit characters even if all the \Cs s in the ingoing sequence are unital, simple and infinite dimensional. 

We define the notion of ``almost characters" and show that the existence of such is related to the invariant $\dv_2( \, \cdot \,)$. It follows in particular that simple unital infinite dimensional \Cs s can have almost characters. 

First we need some technical lemmas:

\begin{lemma} \label{lm:limit}
Let $A$ be a unital \Cs{} and  let $(I_\lambda)$ be an upward directed family  of ideals
of $A$. Set $\overline{\bigcup I_\lambda}=I$. It follows that
$$
\Dv_m(A/I)=\inf_\lambda \Dv_m(A/I_\lambda), \quad 
\ddv_m(A/I)=\inf_\lambda \ddv_m(A/I_\lambda) $$
$$
\dv_m(A/I)=\inf_\lambda \dv_m(A/I_\lambda)
$$
for all positive integers $m$. 
\end{lemma}

\begin{proof}
The inequality ``$\le$" in all three cases follows from Remark~\ref{rem:image} since we have a unital \sh{} $A/I_\lambda\to A/I$ for each $\lambda$. We prove the reverse inequality ``$\ge$" only in the case of $\ddv_m(\, \cdot \,)$; the proofs of the other two instances are similar.

Set $\ddv_m(A/I) = n$, and let us show that $\ddv_m(A/I_\lambda) \le n$ for some $\lambda$. Find $x_1 \dots, x_m$ in $\Cu(A/I)$ be such that 
$$x_1+x_2 + \cdots + x_m \le \langle 1 \rangle \le nx_j$$
for all $j$. Find positive contractions $a_1, \dots, a_m$ in $A \otimes \cK$ such that $x_j = \langle b_j \rangle$, where $b_j \in A/I \otimes \cK$ is the image of $a_j$ under the quotient mapping $A \to A/I$. 
Find $\ep > 0$ such that $\langle 1 \rangle \le n \langle (b_j-\ep)_+ \rangle$ for all $j$. It follows from \cite[Lemma 4.12]{KirRor:pi} that there are positive elements $c,c'_1, \dots, c'_m$ in $I \otimes \cK$ such that
$$\langle (a_1 -\ep/2)_+\rangle + \cdots + \langle (a_m-\ep/2)_+ \rangle \le \langle 1 \rangle + \langle c \rangle, \qquad \langle 1 \rangle \le n \langle (a_j -\ep)_+ \rangle + \langle c'_j \rangle$$
for all $j$. There is $\delta > 0$ such that
$$\langle (a_1-\ep)_+ \rangle + \cdots + \langle (a_m-\ep)_+ \rangle \le \langle 1 \rangle + \langle (c-\delta)_+ \rangle, \quad \langle 1 \rangle \le n \langle (a_j -\ep)_+ \rangle + \langle (c'_j-\delta)_+ \rangle.$$

Since $\bigcup I_\lambda$ is dense in $I$, it follows that $(c-\delta)_+$ and $(c'_j-\delta)_+$ all belong to $I_\lambda \otimes \cK$ for some $\lambda$. Let $z_j \in \Cu(A/I_\lambda)$ be the Cuntz class of the image of the element $(a_j-\ep)_+$ under the quotient mapping $A \to A/I_\lambda$. Then $z_1+ \dots + z_m \le \langle 1 \rangle \le nz_j$, whence $\ddv_m(A/I_\lambda) \le n$.
\end{proof}

\noindent  For each $\ep > 0$, let $h_\ep \colon \R^+ \to [0,1]$ be a continuous functions such that $h_\ep(0) = 0$ and $h_\ep(t) = 1$ when $t \ge \ep$.

\begin{lemma} \label{lm:r}
Let $A$ be a unital \Cs. Let $b_1, b_2, \dots, b_n$ be positive elements in $A$ such that $\sum_{j=1}^n \langle b_j \rangle  \ge \langle 1_A \rangle$. Then, for some $\ep > 0$, there are contractions $y_j$ in $A$ such that $$\sum_{j=1}^n y_j^*h_\ep(b_j)y_j = 1_A.$$
\end{lemma}

\begin{proof} By assumption, and by compactness of $\langle 1_A \rangle$, there are elements $v_j \in A$ such that $\sum_{j=1}^n v_j^*b_{j}v_j = 1_A$. Thus $\sum_{j=1}^n v_j^*(b_{j}-\ep)_+v_j$ is invertible for some $\ep > 0$, and so there are elements $w_j \in A$ such that $\sum_{j=1}^n w_j^*(b_{j}-\ep)_+w_j = 1_A$.   Put $y_j = (b_{j}-\ep)_+^{1/2}w_j$ and 
notice that $h_\ep (b_{j})(b_{j}-\ep)_+ = (b_{j}-\ep)_+$ for all $j$. Thus
$$\sum_{j=1}^n y_j^*h_\ep(b_{j})y_j = \sum_{j=1}^n y_j^*y_j = 1_A,$$
which shows that the $y_j$'s are contractions with the desired properties. 
\end{proof}

\begin{lemma} \label{lm:x}
Let $A$ be a unital \Cs, and let $m,n$ be positive integers. 
\begin{enumerate}
\item $A$ is weakly $(m,n)$-divisible if and only if there exist positive contractions $a_{ij}$ and contractions $y_j$ in $A$, $j=1,2, \dots, n$ and $i=1,2, \dots, m$, such that $a_{1j},a_{2j}, \dots, a_{mj}$ are pairwise equivalent and orthogonal for all $j$, and such that
$1_A = \sum_{j=1}^n y_j^*a_{1j}y_j$.
\item $A$ is $(m,n)$-decomposable if and only if there exist pairwise orthogonal positive contractions $a_i$ and contractions $y_{ij}$ in $A$, $j=1,2, \dots, n$ and $i=1,2, \dots, m$, such that 
$\sum_{j=1}^n y_{ij}^* a_i y_{ij} = 1_A$
for all $i$.
\item $A$ is $(m,n)$-divisible if and only if there exist pairwise equivalent and pairwise orthogonal positive contractions $a_i$ and contractions $y_j$ in $A$,  $j=1,2, \dots, n$ and $i=1,2, \dots, m$, such that $\sum_{j=1}^n y_{j}^* a_1y_{j} = 1_A$.
\end{enumerate}
\end{lemma}

\begin{proof}
We view $A$ as being a sub-\Cs{} of $A \otimes \cK$ which is identified with the corner $1_A(A \otimes \cK) 1_A$. 

(i). ``If". Put $x_j = \langle a_{1j} \rangle = \langle a_{ij} \rangle \in \Cu(A)$. Then 
$$mx_j = \langle \sum_{i=1}^m a_{ij} \rangle \le \langle 1_A \rangle = \big\langle \sum_{j=1}^n y_j^* a_{1j}y_j \big \rangle \le \sum_{j=1}^n \langle y_j^*a_{1j}y_j\rangle \le \sum_{j=1}^n \langle a_{1j} \rangle = \sum_{j=1}^n x_j.$$

``Only if". Choose $x_j' \ll x_j$ such that $\langle 1_A \rangle \le x_1' + x_2' + \cdots + x_n'$. Use Lemma~\ref{lm:comp} (ii) to find positive contractions $b_{ij}$ in $A$, $j=1,2, \dots, n$ and $i=1,2, \dots, m$, such that $b_{1j}, b_{2j}, \dots, b_{mj}$ are pairwise orthogonal and equivalent for all $j$ and such that $x_j' \le \langle b_{ij} \rangle \le x_j$. It then follows that
$1_A \precsim b_{11} \oplus b_{12} \oplus \cdots \oplus b_{1n}$,
and so it follows from  Lemma \ref{lm:r} that there are $\ep > 0$ and contractions $y_j$ in $A$ such that $\sum_{j=1}^n y_j^*h_\ep(b_{1j})y_j =  1_A$. The contractions $y_j$ together with the positive contractions $a_{ij} = h_\ep(b_{ij})$ are then as desired.

(ii). ``If". Put $x_i = \langle a_i \rangle \in \Cu(A)$. Then 
$$x_1+ \cdots + x_m = \big\langle a_1+ \cdots + a_m \big\rangle \le \langle 1_A \rangle = \big\langle  \sum_{j=1}^n y_{ij}^*a_iy_{ij} \big\rangle \le \sum_{j=1}^n \langle y_{ij}^*a_iy_{ij} \rangle \le nx_i$$
for all $i$.

``Only if". Choose $x_i' \ll x_i$ such that $nx_i' \ge \langle 1_A \rangle$. Use Lemma~\ref{lm:comp} (ii) to find pairwise orthogonal and equivalent positive contractions $b_1, b_2, \dots, b_m$ in $A$ such that $x_i' \le \langle b_{i} \rangle \le x_i$. Proceed as in the proof of ``only if" in part (i) to obtain elements $a_i = h_\ep(b_i)$ (for a suitable $\ep > 0$) and $y_{ij}$ with the desired properties. 

(iii). ``If". Put $x = \langle a_1 \rangle = \langle a_i \rangle \in \Cu(A)$. Then
$$mx = \big\langle a_1 + a_2 + \cdots + a_m \rangle \le \langle 1_A \rangle \le \big\langle \sum_{j=1}^n y_j^*a_1y_j \big\rangle \le \sum_{j=1}^n \langle y_j^*a_1y_j \rangle \le nx.
$$

``Only if". Choose $x' \ll x$ such that $\langle 1_A \rangle \le nx'$. By Lemma~\ref{lm:comp} (ii) there are pairwise orthogonal and pairwise equivalent positive contractions $b_1,b_2, \dots,b_m$ in $A$ such that $x' \le \langle b_j \rangle \le x$. We can now follow the proof of ``only if" in part (i) to obtain elements $a_j = h_\ep(b_j)$ (for a suitable $\ep > 0$) and $y_j$ with the desired properties. 
\end{proof}

\noindent
If $(A_k)$ is a sequence of \Cs s, then we denote by $\prod_{k=1}^\infty A_k$ the \Cs{} of all bounded sequences $(a_k)$, with $a_k \in A_k$. If $\omega$ is a (free) filter on $\N$, then denote by $c_\omega(\{A_k\})$ the closed two-sided ideal in  $\prod_{k=1}^\infty A_k$ consisting of those sequences $(a_k)$ for which $\lim_\omega \|a_k\| = 0$. Finally, denote the quotient $\prod_{k=1}^\infty A_k / c_\omega(\{A_k\})$  by $\prod_\omega A_k$. 

\begin{proposition} \label{prop:ultrapower}
Let $(A_n)$ be a sequence of unital \Cs s. Then, for all integers $m \ge 2$ and for any free filer $\omega$ on $\N$ we have:
\begin{enumerate}
\item $\Dv_m\Big(\prod_{k=1}^\infty A_k\Big) = \sup_k \Dv_m(A_k), \quad \Dv_m \Big(\prod_\omega A_k \Big) = \displaystyle{\limsup_{\omega}} \; \Dv_m(A_k)$.
\item $\ddv_m\Big(\prod_{k=1}^\infty A_k\Big) = \sup_k \ddv_m(A_k), \quad
\ddv_m \Big(\prod_\omega A_k \Big) = \displaystyle{\limsup_{\omega}} \; \ddv_m(A_k)$.
\item $\dv_m\Big(\prod_{k=1}^\infty A_k\Big) = \sup_k \dv_m(A_k), \quad
\dv_m \Big(\prod_\omega A_k \Big) = \displaystyle{\limsup_{\omega}} \; \dv_m(A_k)$.
\end{enumerate}
\end{proposition}

\begin{proof} We only prove (i). The proofs of (ii) and (iii) are very similar.

We have unital \sh s $\prod_{k=1}^\infty A_k \to A_n$ for all $n$. Therefore the inequality ``$\ge$" holds in the first identity in (i) (and also in (ii) and (iii)), cf.\ Remark~\ref{rem:image}. 

We show next that $\Dv_m\Big(\prod_{k=1}^\infty A_k\Big) \le \sup_k \Dv_m(A_k)$. Let $n$ be a positive integer such that $\Dv_m(A_k) \le n$ for all $k$. Then, by Lemma \ref{lm:x} (i), for each $k$ we can find positive contractions $a_{ij}^{(k)}$ and contractions $y_j^{(k)}$ in $A_k$, for $j=1,2, \dots, n$ and $i=1,2, \dots, m$, such that $a_{1j}^{(k)}, a_{2j}^{(k)}, \dots, a_{mj}^{(k)}$ are pairwise orthogonal and equivalent for all $j$, and such that
$$1_{A_k} = \sum_{j=1}^n (y_j^{(k)})^*a_{1j}^{(k)} y_j^{(k)}.$$
Put
$$a_{ij} = (a_{ij}^{(k)}) \in \prod_{k=1}^\infty A_k, \qquad y_j = (y_j^{(k)}) \in \prod_{k=1}^\infty A_k.$$
Then $a_{1j},a_{2j}, \dots, a_{mj}$ are pairwise orthogonal and equivalent, and $\sum_{j=1}^n y_j^*a_{1j}y_j$ is equal to the unit of $\prod_{k=1}^\infty A_k$. By Lemma \ref{lm:x} (i), this shows that $\prod_{k=1}^\infty A_k$ is $(m,n)$-divisible, whence $\Dv_m\Big(\prod_{k=1}^\infty A_k\Big) \le n$.

To prove the second part of (i) note first that we have a natural unital (surjective) $^*$-homomorphism $\prod_{k \in I} A_k \to \prod_\omega A_k$ for each $I \in \omega$. We can therefore use Remark~\ref{rem:image}  and the first identity in (i) to conclude that
$$\Dv_m\big(\prod_\omega A_k\big) \le \Dv_m\big(\prod_{k \in I} A_k\big) = \sup_{k \in I} \Dv_m(A_k),$$
which shows that $\Dv_m \big(\prod_\omega A_k \big) \le \limsup_{\omega} \, \Dv_m(A_k)$. 

We proceed to prove the reverse inequality. For each $I \in \omega$ consider the ideal $J(I)$ in $\prod_{k=1}^\infty A_k$ consisting of those sequences $(a_k)$ for which $a_k = 0$ for all $k \in I$. Then
$$c_\omega(\{A_k\}) = \overline{\bigcup_{I \in \omega} J(I)},$$
(where $\omega$ is ordered by reverse inclusion). We can now use Lemma~\ref{lm:limit} and the first identity in (i) to conclude that
\begin{eqnarray*}
\Dv_m\big(\prod_\omega A_k\big) & =& \inf_{I \in \omega} \Dv_m\Big(\big(\prod_{k=1}^\infty A_k\big)/J(I)\Big) \; = \; \inf_{I \in \omega} \Dv_m\big(\prod_{k \in I} A_k\big) 
\\ &= & \inf_{I \in \omega} \sup_{k \in I} \Dv_m(A_k) \; = \; \limsup_\omega \Dv_m(A_k).
\end{eqnarray*}
\end{proof}

\noindent If we combine the proposition above with Corollary~\ref{cor:char} (i) we obtain:

\begin{corollary} \label{cor:dv-char}
Let $(A_k)$  be a sequence of unital \Cs s such that $\lim_{k \to \infty} \dv_2(A_k) = \infty$. Then $\prod_{k=1}^\infty A_k$ has a character, and so does $\prod_\omega A_k$ for each free filter $\omega$ on $\N$. 
\end{corollary}

\noindent If we combine the corollary above with Theorem~\ref{thm:simple}, then we obtain the following surprising fact:

\begin{corollary} \label{cor:simple-char} There is a sequence $(A_k)$ of unital simple infinite dimensional \Cs s such that $\prod_{k=1}^\infty A_k$ and $\prod_\omega A_k$ have characters for each free filter $\omega$ on $\N$.
\end{corollary}

\noindent Clearly, none of the \Cs s $A_k$ in the corollary above can have a character. However, they have "almost characters" in the sense defined below. This is one way of understanding how the product \Cs{} can have a character when none of the individual \Cs s has one. 

\begin{definition} \label{def:N-ep-char}
Let $N \ge 2$ be an integer and let $\ep > 0$. A unital \Cs{} $A$ is said to have \emph{$(N,\ep)$-characters} if for every $N$-tuple $u_1,u_2, \dots, u_N$ of unitaries in $A$ there exists a state $\rho$ on $A$ such that $|\rho(u_j)| \ge 1-\ep$ for $j=1,2, \dots, N$.
\end{definition}

\noindent A state $\rho$ on a unital \Cs{} is a character if and only if $|\rho(u)| = 1$ for all unitary elements $u \in A$. Most simple \Cs s that we know of do not have $(2,\ep)$-characters for small $\ep>0$. For example, if $A$ is a \Cs{} which contains unitaries $u,v$ such that $\|uvu^*v^*- \lambda 1_A\| < \eta$ for some $\lambda \in \T$ and for some $\eta < |1-\lambda|$, then $A$ does not admit any $(2,\ep)$-character for some small enough $\ep > 0$. Indeed, if $\rho$ is a state on $A$ such that $|\rho(u)|$  and $|\rho(v)|$ are close to $1$, then $\rho(uvu^*v^*)$ is close to $1$.

\begin{proposition} \label{prop:n-ep-char}
A unital \Cs{} has a character if and only if it has $(N,\ep)$-characters for all pairs $(N,\ep)$, where $N \ge 2$ is an integer and $\ep > 0$.
\end{proposition}

\begin{proof} The ``only if" part is trivial. Assume that $A$ is a unital \Cs{} that has $(N,\ep)$-characters for all pairs $(N,\ep)$. For each finite subset $F$ of the unitary group of $A$ and for each $\ep > 0$, let $S(F,\ep)$ denote the set of states $\rho$ on $A$ such that $|\rho(u)| \ge 1-\ep$ for all $u \in F$. Then, by assumption, $S(F,\ep)$ is non-empty. It follows that $\bigcap_{(F,\ep)} S(F,\ep)$ is non-empty, and any state in this intersection is a character.
\end{proof}

\begin{proposition} \label{prop:prod-char}
Let $(A_k)$ be a sequence of unital \Cs s, and let $\omega$ be a free ultrafilter on $\N$. Then $\prod_\omega A_k$ has a character if and only if for each integer $N \ge 2$ and for each $\ep > 0$ there exists $I \in \omega$ such that $A_k$ has $(N,\ep)$-characters for each $k \in I$.
\end{proposition}

\begin{proof} We prove first the ``if" part. By Proposition \ref{prop:n-ep-char} it suffices to show that $\prod_\omega A_k$ has $(N,\ep)$-characters for all $(N,\ep)$. Fix $(N,\ep)$ and find $I \in \omega$ such that $A_k$ has $(N,\ep)$-characters for each $k \in I$. Let $u_1, \dots, u_N$ be unitaries in $\prod_\omega A_k$, and let $(u_j^{(k)}) \in \prod_{k=1}^\infty A_k$ be a lift of $u_j$. Then for each $k \in I$ there is a state $\rho_k$ on $A_k$ such that $|\rho_k(u_j^{(k)})| \ge 1 - \ep$ for $j=1,2, \dots, N$. Choose arbitrary states $\rho_k$ on $A_k$ for $k \notin I$ and define a state $\rho$ on $\prod_\omega A_k$  by $\rho(x) = \lim_\omega \rho_k(x_k)$, where $(x_k) \in \prod_{k=1}^\infty A_k$ is a lift of $x$. (A priori, $\rho$ defines a state on $\prod_{k=1}^\infty A_k$, and one checks that it vanishes on the ideal $c_\omega(\{A_k\})$.) Then 
$$|\rho(u_j)| = \lim_{k \to \omega} |\rho_k(u_j^{(k)})| \ge \inf_{k \in I}  |\rho_k(u_j^{(k)})| \ge 1 - \ep,$$
for $j=1,2, \dots, N$, which shows that $\prod_\omega A_k$ has $(N,\ep)$-characters.

Suppose next that $\prod_\omega A_k$ has a character $\rho$. Fix $(N,\ep)$, and let $J$ be the set of those $k \in \N$ for which $A_k$ does not have $(N,\ep)$-characters. For each $k \in J$ choose unitaries $u_j^{(k)}$ in $A_k$, $j=1,2, \dots, N$, such that there is no state $\rho'$ on $A_k$ for which 
$|\rho'(u_j^{(k)})| \ge 1-\ep$ for all $j=1,2, \dots, N$. Choose arbitrary unitaries $u_j^{(k)} \in A_k$ for $k \notin J$, and let $u_j$ be the unitary element $(u_j^{(k)})$ in $\prod_{k=1}^\infty A_k$. Let $B$ be the (separable) sub-\Cs{} of $\prod_\omega A_k$ generated by the unitaries $\pi_\omega(u_j)$, where $\pi_\omega$ is the quotient mapping $\prod_{k=1}^\infty A_k \to \prod_\omega A_k$. By \cite[Lemma 2.5]{Kir:Abel} there is a sequence $\rho_k$ of pure states on $A_k$ such that $\rho(\pi_\omega(x)) = \lim_\omega \rho_k(x_k)$ for all $x = (x_k) \in \prod_{k=1}^\infty A_k$ with $\pi_\omega(x) \in B$. Now, 
$$1 = |\rho(\pi_\omega(u_j))| = \lim_\omega |\rho_k(u_j^{(k)})| = \liminf_{k \to \omega}  |\rho_k(u_j^{(k)})|  = \sup_{I \in \omega} \inf_{k \in I}  |\rho_k(u_j^{(k)})|.$$
It follows that there exists $I \in \omega$ such that $ |\rho_k(u_j^{(k)})| \ge 1 - \ep$ for all $k \in I$ and for all $j=1,2, \dots, N$. This entails that $I \cap J = \emptyset$. Hence $A_k$ has $(N,\ep)$-characters for all $k \in I$. 
\end{proof}

\noindent We can relate the existence of $(N,\ep)$-characters on a \Cs{} $A$ to the divisibility quantity $\dv_2(A)$. 

\begin{theorem} For each pair $(N,\ep)$, where $N \ge 2$ is an integer and $\ep >0$, there exists an integer $n \ge 2$ such that every unital\Cs{} $A$ which satisfies $\dv_2(A) \ge n$ has $(N,\ep)$-characters. Conversely, for every integer $n \ge 2$ there exists a pair $(N,\ep)$, where $N \ge 2$ is an integer and where $\ep > 0$, such that every unital \Cs{} $A$ which has $(N,\ep)$-characters satisfies $\dv_2(A) \ge n$. 
\end{theorem}

\begin{proof}  Suppose that the first claim were false. Then there would exist a pair $(N,\ep)$ and a sequence $(A_n)$ of unital \Cs s such that $\dv_2(A_n) \ge n$ and none of the $A_n$'s have $(N,\ep)$-characters. However, if $\omega$ is any free ultrafilter $\N$, then $\prod_\omega A_n$ has a character by Corollary \ref{cor:dv-char}, whence $A_n$ has $(N,\ep)$-characters for each $n$ in some subset $I \in \omega$, a contradiction.

Suppose next that the second statement were false. Then there would exist an integer $n \ge 2$ and a  sequence $(A_k)$ of unital \Cs s such that $A_k$ has $(k,1/k)$-characters but $\dv_2(A_k) < n$. Let $\omega$ be a free ultrafilter on $\N$. It then follows from Proposition \ref{prop:prod-char} that $\prod_\omega A_k$ has a character. Hence $\dv_2(\prod_\omega A_k) = \infty$, whence $\lim_\omega \dv_2(A_k) = \infty$ by Proposition \ref{prop:ultrapower}, a contradiction. 
\end{proof} 

\begin{corollary} For each pair  $(N,\ep)$, where $N \ge 2$ is an integer and $\ep >0$, there exists a unital simple infinite dimensional \Cs{} which has $(N,\ep)$-characters. 
\end{corollary}

\noindent We end this section by giving several equivalent formulation of some well-known open problems for \Cs s. Recall that a \Cs{} $A$ has the Global Glimm Halving property if there is a $^*$-homomorphism $CM_2(\C) \to A$ whose image is full in $A$.

\begin{proposition} \label{prop:GGH3}
The following statements  are equivalent:
\begin{enumerate}
\item Every unital \Cs{} that has no finite dimensional representation has
the Global Glimm Halving property.
\item For all unital \Cs s $A$, if $\dv_m(A) < \infty$ for all $m \ge 2$, then $\Dv_2(A) < \infty$.
\item For every sequence $(A_k)$ of unital \Cs s, if $\sup_k \dv_m(A_k) < \infty$ for all $m$, then $\sup_k \Dv_2(A_k) < \infty$.
\end{enumerate}
\end{proposition}

\begin{proof} (i) $\Leftrightarrow$ (ii). $A$ has no finite dimensional representations if and only if $\rnk(A) \ge m$ for all $m$, which by Corollary~\ref{cor:finite-div} (iii) is equivalent to $\dv_m(A) < \infty$ for all $m$. 
It was shown in Corollary~\ref{cor:finite-div} (i) that the Global Glimm Halving property holds for $A$ if and only if $\Dv_2(A) < \infty$.

(ii) $\Rightarrow$ (iii). Given a sequence $(A_k)$ of unital \Cs s such that $\sup_k \dv_m(A_k) < \infty$ for all $m$. Consider the \Cs{} $A = \prod_{k=1}^\infty A_k$. Then $\dv_m(A) = \sup_k \dv_m(A_k) < \infty$ by Proposition \ref{prop:ultrapower}. Thus $\Dv_2(A) < \infty$, which implies that $\sup_k \Dv_2(A_k) = \Dv_2(A) < \infty$, again by Proposition~\ref{prop:ultrapower}. 

(iii) $\Rightarrow$ (ii) is trivial: Take $A_k = A$ for all $k$.
\end{proof}

\begin{proposition}  \label{prop:GGH4}
The following statements  are equivalent:
\begin{enumerate}
\item All unital \Cs s $A$ that have no finite dimensional representation contain two positive full elements that are orthogonal to each other. 
\item For all unital \Cs s $A$, if $\dv_m(A) < \infty$ for all $m \ge 2$, then $\ddv_2(A) < \infty$.
\end{enumerate}
\end{proposition}

\begin{proof} (i) $\Leftrightarrow$ (ii). As in the proof of Proposition \ref{prop:GGH3}, 
$A$ has no finite dimensional representations if and only if $\dv_m(A) < \infty$ for all $m$. 
It was shown in Corollary~\ref{cor:finite-div} (ii) that $A$ contains two positive full elements that are orthogonal to each other if and only if $\ddv_2(A) < \infty$.

(ii) $\Rightarrow$ (iii) is similar to the proof of (ii) $\Rightarrow$ (iii) in Proposition \ref{prop:GGH3}. 
(iii) $\Rightarrow$ (ii) is trivial.
\end{proof}

\section{Infinite elements} \label{sec:infinite}

\noindent
Following \cite{KirRor:pi}, a Cuntz class $u$ in the Cuntz semigroup of a \Cs{} $A$ is said to be properly 
infinite if it satisfies  $u=2u$ (whence $u=\infty \! \cdot  \! u$). Similarly, a countably generated Hilbert module over $A$ is properly infinite if its  Cuntz class is properly infinite. We saw in Proposition~\ref{prop:propinf}  how infiniteness of a Cuntz class can arise from a certain type of divisibility property. In this section we shall investigate this and related phenomena further with emphasis on the following property:

\begin{definition} \label{def:omega-n}
Let $A$ be a \Cs, let $n \ge 1$ be an integer, and let $u$ be an element in $\Cu(A)$. We say that $u$ is $(\omega,n)$-decomposable
if there exist $x_1,x_2,\dots$ such that $\sum_{i=1}^\infty x_i\leq u$ and $u \le nx_i$ for all $i$.
\end{definition}

\noindent If $u$ is $(\omega,n)$-decomposable, then  $u$ is $(m,n)$-decomposable for all $m$. In particular, by Proposition~\ref{prop:propinf} (ii), it follows that $nu$ is properly infinite. 

The condition in the definition above can be reformulated in several different ways:

\begin{lemma}\label{rem:indexpartition}
Let $A$ be a \Cs, let $n \ge 1$ be an integer, and let $u$ be an element of $\Cu(A)$. Then the following conditions are equivalent:
\begin{enumerate}
\item $u$ is $(\omega,n)$-decomposable,
\item there exist $x_1,x_2,\dots$ such that $\sum_{i=1}^\infty x_i\leq u$ and $nx_i = \infty \! \cdot \! u$ for all $i$,

\item there exist $x_1,x_2,\dots$ and $y_1,y_2,\dots$ such that  $\sum_{i=1}^\infty x_i\leq u$, 
$y_{i-1}\leq y_i\leq nx_i$ for all $i$, and $\infty \! \cdot \! u\leq \infty \! \cdot \! \sup y_i$,

\item there exist $x_1,x_2,\dots$ such that $\sum_{i=1}^\infty x_i\leq u$ and $n\sum_{j=k}^\infty x_j = \infty \! \cdot \! u$ for all $k$.
\end{enumerate}
\end{lemma}

\begin{proof} (i) $\Rightarrow$ (ii). Suppose that $x_1,x_2, \dots$ satisfy $\sum_{i=1}^\infty x_i \le u \le nx_j$. Let $\{I_i\}_{i=1}^\infty$ be a partition of the natural numbers into infinite sets. Then the elements $x_i' = \sum_{j \in I_i} x_j$ witness that condition (ii) holds. 

To get (ii) $\Rightarrow$ (iii), set $y_i=\infty \! \cdot \!u$ for all $i$ and choose $(x_i)_{i=1}^\infty$ that satisfies  (ii). 

(iii)$\Rightarrow$ (iv).  Let $(x_i)_{i=1}^\infty$ and $(y_i)_{i=1}^\infty$ be as in (iii).
Then $\sum_{i=k}^\infty x_i\geq \infty \! \cdot \! y_k$. Observe that the left side of this inequality
decreases as $k$ increases. Thus, $\sum_{i=k}^\infty x_i\geq \infty \! \cdot \! y_{k'}$ for all $k'\geq k$.
Taking the supremum over all $k'\geq k$ we get 
$\sum_{i=k}^\infty x_i\geq \infty \! \cdot \! \sup y_i\geq \infty \! \cdot \!  u$.

 (iv) $\Rightarrow$ (i). Suppose that $x_1,x_2, \dots$ satisfy the condition in (iv). Let $(u_i)_{i=1}^\infty$ be such that $u_i\ll u_{i+1}$ for all $i$ and $\sup_i u_i=u$. Then there exists a sequence $1 = k_0 < k_1 < \cdots$ such that the elements $x_i' = \sum_{j=k_{i-1}}^{k_{i}-1} x_j$ satisfy $x_i'\geq u_i$ for all $i$. Let
$\{I_i\}_{i=1}^\infty$ be a partition of the natural numbers into infinite sets.
Then the elements $x_i'' = \sum_{j \in I_i} x_j'$ satisfy  the condition in Definition~\ref{def:omega-n}. 
\end{proof}
%
%
%

\noindent It was shown in \cite{OPR:Cuntz} that the Corona Factorization Property for a \Cs{} is equivalent to a condition for its Cuntz semigroup, that we here shall refer to as (CFP4S).  A complete ordered abelian semigroup is said to have (CFP4S) if whenever $(x_i)_{i=1}^\infty$ is a full sequence, $x'\ll x_1$, and $(y_i)_{i=1}^\infty$ is such that $my_i\geq x_i$ for all $i$ and some $m$, then there exists $n$ such that  $\sum_{i=1}^n y_i\geq x'$. Recall that a full sequence is one that is increasing and such that $\sup x_i$ is a full element, i.e., $\infty \! \cdot  \! \sup x_i$ is the largest element of the semigroup (which we shall denote by $\infty$).

In Section 6 we discussed a related notion, called the strong Corona Factorization Property, and its analog for the Cuntz semigroup.

The proposition below relates the (CFP4S) with the notion of
$(\omega,m)$-divisibility. In fact, it is a consequence of this proposition
that a semigroup in the category $\CCu$ has a full elements which is
$(\omega,m)$-divisible and not properly infinite if and only if the
semigroup does not satisfy (CFP4S).

\begin{proposition} \label{prop:CFP4S}
The following four conditions are equivalent for any object $S$ in the category $\CCu$. 
\begin{enumerate}
\item $S$ has property (CFP4S).
\item For every sequence $(y_i)_{i=1}^\infty$ in $S$, if there is a full sequence $(x_i)_{i=1}^\infty$ in $S$ such that $my_i \ge x_i$ for some $m$ and for all $i$, then $\sum_{i=1}^\infty y_i = \infty$. 
\item For every sequence $(y_i)_{i=1}^\infty$ in $S$,  if $my_i=\infty$
for some $m$ and for all $i$, then $\sum_{i=1}^\infty y_i=\infty$.
\item  For every sequence $(y_i)_{i=1}^\infty$ in $S$, if $\sum_{i=n}^\infty my_i=\infty$
for some $m$ and for all $n$, then $\sum_{i=1}^\infty y_i=\infty$.
\item For every full element $y$ in $S$, if $y$ is $(\omega,m)$-decomposable for some $m$, then $y$ is properly infinite (whence $y = \infty$). 
\end{enumerate}
\end{proposition}

\begin{proof} (i) $\Rightarrow$ (ii). Apply the (CFP4S) to the tail sequences $(x_i)_{i=n}^\infty$ and $(y_i)_{i=n}^\infty$. Then we get $x' \le \sum_{i=n}^N y_i \le \sum_{i=n}^\infty y_i$ for all $x' \ll x_n$, whence $\infty = \sup_n x_n \le  \sum_{i=1}^\infty y_i$.

(ii) $\Rightarrow$ (i). If $(x_i)$ is a full sequence and if $(y_i)$ is another sequence such that $x_i \le my_i$, then $\sum_{i=1}^\infty y_i = \infty$ by (ii). In particular, if $x' \ll x_1$, then $x' \le \sum_{i=1}^n y_i$ for some $n$ by the definition of compact containment.

(ii) $\Rightarrow$ (iii). Suppose that (ii) holds and that $(y_i)$ is a sequence in $S$ such that $my_i = \infty$ for all $i$.  Then $\infty = \sum_{i=1}^\infty y_i$ by (ii) with $x_i = \infty$ for all $i$. 

(iii)  $\Rightarrow$ (v). Suppose that (iii) holds, and that $y \in S$ is full and $(\omega,m)$-decomposable. Then $y$ satisfies condition (ii) of  Lemma~\ref{rem:indexpartition}, so there exists a sequence $(y_i)$ such that $\sum_{i=1}^\infty y_i \le y$ and $my_i = \infty \! \cdot \! y = \infty$ for all $i$. But then $\sum_{i=1}^\infty y_i = \infty$ because (iii) holds, whence $y  = \infty$.

(v)  $\Rightarrow$ (iv). Suppose that (v) holds and let $(y_i)_{i=1}^\infty$ be a sequence in $S$ such that $\sum_{i=n}^\infty my_i=\infty$ for some $m$ and for all $n$. Put $y = \sum_{i=1}^\infty y_i$. We must show that $y = \infty$. We know that $my = \infty$, so $y$ is full. It is easy to see that $y$ satisfies condition (iv) of Lemma~\ref{rem:indexpartition}, so $y$ is $(\omega,m)$-decomposable. Hence $y= \infty$ by the assumption that (v) holds. 

(iv) $\Rightarrow$ (ii). Let $(x_i)_{i=1}^\infty$ be a full sequence in $S$, let $m \ge 1$ be a positive integer, and let $(y_i)$ be such that $my_i\geq x_i$ for all $i$. Then 
$$\sum_{i=n}^\infty my_i\geq \sum_{i=n}^\infty x_i \geq \sum_{i=k}^\infty x_i \ge \infty \! \cdot \! x_k$$
for all $k \ge n$.  As $\infty = \sup_k \infty \! \cdot \! x_k$, we conclude that  $\sum_{i=n}^\infty my_i=\infty$ for all $n$. By (iv) this entails that $\sum_{i=1}^\infty y_i=\infty$, and in particular that $x_1 \le \sum_{i=1}^\infty y_i$. 
\end{proof}

\noindent In the following example we describe a Cuntz semigroup with an element $u$ that is $(\omega,2)$-decomposable but not properly infinite. In particular, $2u$ is properly infinite while $u$ is not. This example is well known in other contexts, and it was discussed in the paragraph preceding Corollary~\ref{cor:key}.

\begin{example}\label{ex:infty2decomposable}
Let $X=(S^2)^\infty$ be a countable cartesian product of 2-dimensional spheres, and let $p_i \in C(X, \cK)$ be the one-dimensional projection arising as the pull back of a non-trivial rank one projection $p$ in $C(S^2) \otimes \cK$ along the $i$th coordinate projection $X \to S^2$, cf.\ the comments above Corollary~\ref{cor:key}. Let $e$ be a trivial one-dimensional projection. Then $e \nprecsim \bigoplus_{i=1}^N p_j$ for all $N$, because the Euler class of the projection on the right-hand side is non-trivial. 

Put $x_i = \langle p_i \rangle$, $v = \langle e \rangle$, and put $u = \sum_{i=1}^\infty x_i$. Then $v \nleq u$,  $v \le 2x_i$, and $u+u = \infty \! \cdot \! v= \infty \! \cdot \! u$. Hence $2\sum_{j=i}^\infty x_i =  \infty \! \cdot \! v= \infty \! \cdot \! u$, and so $u$ is $(\omega,2)$-decomposable; but $u$ is not pro\-per\-ly infinite. 
\end{example}

\noindent
We now look more closely at the properties of $(\omega,n)$-decomposable elements. 

\begin{proposition}\label{prop:tensorstability}
Let $(a_i)_{i=0}^\infty$ be a sequence of mutually orthogonal positive elements in a \Cs{} $A$ such that 
$\sum_{i=0}^\infty a_i$ converges to a strictly positive element in $A$. Assume that $\sum_{j\geq i} n\langle a_i\rangle=\infty$ for all $i$. Then $A\otimes B$ is stable for every $\sigma$-unital \Cs{} $B$ with $\rank(B)\geq n$.
\end{proposition}

\begin{proof} 
Set $\sum_{i=0}^\infty a_i=a\in A$ and let $b$ be a strictly positive element in $B$. 
Notice that $a\otimes b$ is a strictly positive element of $A\otimes B$.
In order to prove stability of $A\otimes B$ we will use the stability criterion obtained in \cite{HjeRor:stable}: $A\otimes B$ is stable if for every $\ep>0$
there exists a positive element $c$ in $A$ which is orthogonal to $(a\otimes b-\ep)_+$ and satisfies $\langle(a\otimes b-\ep)_+\rangle\leq \langle c\rangle$.

Arguing as in the proof of  (iv) $\Rightarrow$ (i) in Lemma \ref{rem:indexpartition} we may assume that $n\langle a_i\rangle=\infty$
for all $i$. By Theorem \ref{thm:finite-div} (iii), $\rank(B)\geq n$ is equivalent to  weak 
$(n,\omega)$-divisibility for $\langle b\rangle$. Thus there exist a sequence $(x_i)_{i=1}^\infty$ in $\Cu(B)$
such that $nx_i\leq \langle b\rangle$ for all $i$ and $\sum_{i=1}^\infty x_i=\infty$. We can form a new sequence $(x_i')_{i=1}^\infty$ in which each $x_i$ appears repeated infinitely often. 
In this way we may assume without loss of generality  that 
$\sum_{i\geq j}^\infty x_i=\infty$ for all $j$. Find positive elements $b_i$ in $B$ such that $x_i =\langle b_i\rangle$ and $\|b_i\|\leq 2^{-i}$. 
Then
\[
(a\otimes b-\ep)_+=\sum_{i=1}^\infty (a_i\otimes b-\ep)_+=
\sum_{i=1}^N (a_i\otimes b-\ep)_+,  
\]
for some integer $N \ge 1$. Set $c=\sum_{i>N} a_i\otimes b$. Then $c$ is orthogonal to $(a\otimes b-\ep)_+$. Also, 
\[
\langle a_i\otimes b\rangle=
\langle a_i\rangle\otimes \langle b\rangle\geq n \,
\langle a_i\rangle\otimes \langle b_i\rangle=\infty \! \cdot  \! \langle a\rangle\otimes \langle b_i\rangle
\]
for each $i$. Hence
\[
\langle c\rangle =\sum_{i>N} \langle a_i\otimes b\rangle=\sum_{i>N} \infty \!\cdot \! \langle a\rangle\otimes \langle b_i\rangle=\infty.
\]
Thus, $\langle(a\otimes b-\ep)_+\rangle\leq \infty=\langle c\rangle$. This shows that $A\otimes B$ is stable.
\end{proof}

\noindent 
The proposition above can be applied to the \Cs{} $A = P(C(X) \otimes \cK)P$ arising from Example \ref{ex:infty2decomposable} with $P = \bigoplus_{i=1}^\infty p_i \in \cM(C(X) \otimes \cK)$. The \Cs{} $A$ is not stable (because $e \notin A$ while $e \in M_2(A)$), but $A \otimes B$ is stable for every \Cs{} $B$ that does not have a character by  Proposition~\ref{prop:tensorstability} and 
Example~\ref{ex:infty2decomposable}.

The example obtained  in \cite{Ror:sns} of a simple \Cs{} $A$ of stable rank 1 such that $M_n(A)$ is stable, while $M_{n-1}(A)$ is not stable, likewise satisfies the hypotheses of Proposition \ref{prop:tensorstability}. In fact, to the authors knowledge, every example of a \Cs{} that tensored with $M_n(\C)$ becomes stable also has the stronger property of becoming stable after being tensored with any \Cs{} that has no representations of dimension less than $n$. This raises the following question: 

\begin{question}
Is there a \Cs{}  $A$ such that $M_2(A)$ is stable but $A\otimes B$ is not stable for some \Cs{} $B$ without 
characters?
\end{question}

\begin{proposition} The following statements are equivalent for every \Cs{} $A$ with a strictly positive element $a$.
\begin{enumerate}
\item
$\langle a\rangle$ is $(\omega,n)$-decomposable.

\item
$A$ contains a full hereditary subalgebra $B$ such that $B\otimes C$ is stable for any  
\Cs{} $C$ such that $\rank(C)\geq n$.

\item 
$A$ contains a full hereditary subalgebra $B$ such that $M_n(B)$ is stable.
\end{enumerate}
\end{proposition}

\begin{proof}
(i) $\Rightarrow$ (ii). Let $(x_i)_{i=1}^\infty$ be such that $\sum_{i=1}^\infty x_i\leq \langle a\rangle$ and $nx_i=\infty$. Let $b\in A\otimes \mathcal K$ be strictly positive and 
let $b_i\in A\otimes \mathcal K$ be mutually orthogonal
elements such that $\langle b_i\rangle=x_i$. We can find mutually orthogonal positive elements $a_i$ in $A$ such that $\langle a_i\rangle\leq \langle b_i\rangle$, 
$n\langle a_i\rangle\geq \langle(b-1/ i)_+\rangle$, and such that $\sum_{i=1}^\infty a_i$ is convergent. It then follows from Proposition~\ref{prop:tensorstability}  that (ii) holds when $B$ is the hereditary sub-\Cs{} generated by $\sum_{i=1}^\infty a_i$. 

(ii) $\Rightarrow$ (iii) is clear.

(iii) $\Rightarrow$ (i). Since $A$ is $\sigma$-unital, and $B$ is stably isomorphic to $A$, $B$ is $\sigma$-unital too. Let $b$ be a strictly positive element in $B$. 

Use \cite[Lemma 5.3]{OPR:Cuntz} to find a sequence $(b_k)$ of pairwise orthogonal positive elements in $B$ such that $\langle (b-1/k)_+ \rangle \le n\langle (b-1/k)_+ \rangle \le n \langle b_k \rangle$ for all $k$. Then condition (iii) of  Lemma~\ref{rem:indexpartition} is satisfied with $u = \langle b \rangle$, $x_k = \langle b_k \rangle$, and $y_k = \langle (b-1/k)_+ \rangle$, whence $\langle b \rangle$ is $(\omega,n)$-decomposable. 

Finally, by the fact that $\langle b\rangle\leq \langle a\rangle\leq \infty\! \cdot  \!\langle b\rangle$, it follows by the equivalence of (i) and (ii) in Lemma~\ref{rem:indexpartition} that $\langle a \rangle$ is $(\omega,n)$-decomposable.
%
\end{proof}

\begin{definition}
Let $n\in \N$ and $u\in \Cu(A)$. We call $u$ weakly $(\omega,n)$-divisible if for every $u'\ll u$ there exist $x_i\in \Cu(A)$, $i=1,2,\dots,n$, such that $\infty\! \cdot  \! x_i\leq u$ and 
$u'\leq \sum_{i=1}^n x_i$.
\end{definition}

\noindent Observe that if $u$ is weakly $(\omega,n)$-divisible, then $u$ is weakly $(m,n)$-divisible for all $m\in \N$, whence $nu$ is properly infinite by Proposition \ref{prop:propinf}.

We will next give an example of a Cuntz semigroup element that is weakly $(\omega,2)$-divisible but not properly infinite.  This example needs some preparatory results. 
Let us first recall an example given by Dixmier and Doaudy in \cite{DixDou:champs}.

\begin{example}[Dixmier--Douady, {\cite[\S 17]{DixDou:champs}}] \label{ex:DixDou} 
Let $B_{\infty}$ denote the closed unit ball of $l_2(\N)$ endowed with the weak topology.
Let $l_2(B_\infty)$ denote the $C(B_\infty)$-Hilbert module  of continuous maps
from $B_\infty$ to $l_2(\N)$. We will construct a 
countably generated $C(B_\infty)$-Hilbert  module $D$ such that $l_2(B_\infty)\hookrightarrow D\hookrightarrow l_2(B_\infty)$
but $D\ncong l_2(B_\infty)$.  

Let $x\colon B_\infty\to l_2(\N)\oplus \C  e$ be given by 
\[
x(z)=z+\sqrt{1-\|z\|^2}\cdot e,\hbox{ for }z\in B_\infty.
\] 
Consider the $C(B_\infty)$-module  $D_0$ of functions from $B_\infty$ to $l_2(\N)\oplus \C e$ that have the form $y+x\lambda$, with $y\in l_2(B_\infty)$
and $\lambda\in C(B_\infty)$. The module $D_0$ is a pre-Hilbert C*-module over $C(B_\infty)$ when endowed with the pointwise inner product. Indeed, if $y_1+x\lambda_1$
and $y_2+x\lambda_2$ are vectors in $D_0$ then 
\[
\langle y_1+x\lambda_1,y_2+x\lambda_2\rangle=\langle y_1,y_2\rangle+
\langle y_1,z\rangle\lambda_2+\overline{\lambda_1}\langle z,y\rangle+\overline{\lambda_1}\lambda_2\in C(B_\infty).
\]
Let $D$ denote the completion of $D_0$ with respect to the norm induced by its $C(B_\infty)$-valued inner product. 
Observe that $l_2(B_\infty)\hookrightarrow D_0\subseteq D$. Since $D$ is countably generated, we also have that $D\hookrightarrow l_2(B_\infty)$
by Kasparov's stabilization theorem. Let us see that $D\ncong l_2(B_\infty)$. 
Consider $E\subseteq D$, the orthogonal complement of $\{x\}$. Then $E=\overline{E_0}$, where
\begin{align}\label{Emodule}
E_0=\{y+x\lambda \in D\mid \langle y(z),z\rangle+\lambda(z)=0\, \hbox{ for all }z\in B_\infty\}.
\end{align}
It was implicitly shown by Dixmier and Douady, and explicitly pointed out by Blanchard and Kirchberg (\cite[Proposition 3.6]{BlanKir:Glimm}), that
 for any $v\in E$ there exists $z\in B_\infty$ such that $\langle v,v\rangle (z)=0$. That is, every section of $E$
vanishes at some point (we will reprove this fact in Proposition \ref{prop:sectionsvanish} below). Notice that 
\[
D=E+x\cdot C(B_\infty)\cong E\oplus C(B_\infty).
\]
It can be deduced from this that $D\ncong l_2(B_\infty)$ (see \cite[Proposition 19]{DixDou:champs}).
\end{example}

Let $B_3$ denote the unit ball in $\R^3$. Let $f\in M_2(B_3)^+$ be defined as
\[
f(x,y,z)=\frac{1}{2}
\begin{pmatrix}
1+z & x-iy\\
x+iy & 1-z
\end{pmatrix}.
\]
(The function $f$ is a homeomorphism from $B_3$ to the set of positive elements of $M_2(\C)$ with trace 1. On the boundary 2-sphere of $B_3$ it agrees with the tautological rank 1 projection.) Consider the $C(B_3)$-module associated to $f$:
\begin{align}\label{moduleF}
F:=\overline{f 
\begin{pmatrix}
C(B_3)\\
C(B_3)
\end{pmatrix}
} .
\end{align}

\begin{proposition}\label{prop:sectionsvanish}
Let $B_\infty$ and $B_3$ be as before. Let  $X=\prod_{i\in I} X_i$, where each $X_i$ is either $B_\infty$ or $B_3$
and the index set $I$ is non-empty. For each $i$, let $H_i$ be the pull-back along the projection map $\pi_i\colon X\to X_i$ of either the module $E$ defined  in Example \ref{ex:DixDou} or the module $F$ defined in \eqref{moduleF}. Finally, let $H$ be the $C(X)$-module defined by
$H=\bigoplus_{i\in I} H_i$.
Then $C(X)$ does not embed in $H$ as a $C(X)$-module (i.e., for every $v\in H$ there exists $z\in X$ such that $\langle v,v\rangle (z)=0$).
\end{proposition}

\noindent
Notice that if every $X_i$ agrees with $B_3$, the above proposition can be proven using standard methods in algebraic topology (e.g., characteristic classes).
Indeed, it suffices to restrict to the boundary 2-sphere of each $X_i$ and use that on that set $F$ is the tautological rank 1 projective module.
It is the inclusion of the spaces $B_\infty$ in the definition of $X$ that forces us to use a different route in the proof.  

\begin{proof}
Let $v\in H$, and write $v=\sum_{i\in I} v_i$, with $v_i\in H_i$.
In order  to show $\langle v,v\rangle(z)=0$ for some $z\in X$, it suffices to prove this for  $v$ belonging to  a dense submodule of $H$.
For suppose that  $(v^{(n)})$ is a sequence in $H$ such that $v^{(n)}\to v$ and  $\langle v^{(n)},v^{(n)}\rangle (z_n)=0$ for some $z_n\in X$. Then by the compactness of $X$
there exists a subsequence $(z_{n_k})$ such that  $z_{n_k}\to z\in X$, and so $\langle v,v\rangle (z)=0$.
Thus, we may assume that the index set $I$ is finite. Furthermore, for the indices $i$ such that 
$H_i=\pi_i^*(E)$, we may assume that $v_i\in H_i'$, where  $H_{i}'\subseteq H_i$ is the pull back along $\pi_i$ of the dense submodule $E_0$ defined in \eqref{Emodule}.

In the sequel, we assume that $I=\{1,2,\dots,n\}$, $X_i=B_\infty$ for $i=1,2,\dots,n_1$, 
and $X_i=B_3$ for $i=n_1+1,\dots,n$, where $n_1\leq n$.

We will argue by contradiction that $\langle v,v\rangle(z)=0$ for some $z\in X$.
Suppose that $\langle v,v\rangle$ is invertible, and assume without loss of generality 
that $\langle v,v\rangle=1$. Observe that, for each $i\leq n_1$, $v_i$ is a function from $X$ into the unit ball of $l_2(\N)\oplus \C$, while for $n_1<i\leq n$ the entry $v_i$ is a function from $X$ into  the unit ball of $\C\oplus \C$ (let us denote it by $B_4$). Let $h_0\colon l_2(\N)\oplus \C\to l_2(\N)$ denote the projection onto the first direct summand 
and let $h_1\colon B_4\to B_3$ denote the Hopf fibration (extended to the unit ball):
\[
h_1(z_0,z_1):=(2z_0\overline{z_1},|z_0|^2-|z_1|^2).
\]
Let $\lambda\colon [0,1]\to [0,1]$ be such that $\lambda(0)=0$, $\lambda(t)=1$ for $t\in [\frac 1 n,1]$,  and $\lambda$ is linear in $[0,\frac 1 n]$. 
Define $\tilde h_0,\tilde h_1\colon B_4\to B_3$ by
\begin{align*}
\tilde h_0(w) =h_0\Big(\frac{\lambda(w)}{|w|}w\Big),\quad
\tilde h_1(w) =-h_1\Big(\frac{\lambda(w)}{|w|}w\Big).
\end{align*}
Consider the continuous  map $\Phi\colon X\to X$ given by the vector of functions
\[
\Phi:=(\tilde h_0\circ v_1,\tilde h_0\circ v_2,\dots, \tilde h_0\circ v_{n_1},\tilde h_1\circ v_{n_1+1},\dots,\tilde h_1\circ v_n).
\]
Since $X$ is a compact convex subset of the vector space $(l_2(\N))^{n_1}\times (\R^3)^{n-n_1}$,  the map $\Phi$
 has a fixed point by the Schauder fixed point theorem.  Let $\tilde z:=(\tilde z_i)_{i=1}^n\in X$ be a fixed point of $\Phi$.
Since $\|v(\tilde z)\|=1$, we must have $\|v_i(\tilde z)\|\geq \frac{1}{n}$ for at least one index $i$. Notice that both $\tilde h_{0}$ and $\tilde h_1$
map all vectors of norm at least $1/n$ into the unit sphere of either $B_\infty$ or $B_3$. It follows that the fixed point
$\tilde z$ satisfies  $\|\tilde z_i\|=1$ and $\tilde z_j=0$ for all $j\neq i$. 

There are two cases to consider: $i\leq n_1$ and $i> n_1$.
Suppose that $i\leq n_1$. The general form of $v_i\in H_i'$ is $f+(z_i+\sqrt{1-\|z_i\|^2}e)\alpha $, for some $f\colon X\to l_2(\N)$ and $\alpha\in C(X)$. Since $\|\tilde z_i\|=1$, we have  
\[v_{i}(\tilde z)=f(\tilde z)+\alpha(\tilde z) \tilde z_i=\tilde z_i.\] 
But $\langle f(\tilde z),\tilde z_{i}\rangle + \alpha(\tilde z)=0$.
This contradicts that $\|\tilde z_i\|=1$.   

Suppose that $i>n_1$. Since $z_i\mapsto v_i(\cdots,z_i,\cdots)$, with $z_i\in S^2$, is a section of the tautological bundle on $S^2$, we have $h_1\circ v_i(z)=z_i$ whenever $z_i\in S^2$.  It follows
that $\tilde h_1\circ v_i(\tilde z)=-\tilde z_i$. But $\tilde h_1\circ v_i(\tilde z)=\tilde z_i$, by the fixed point property of $\tilde z$. This again contradicts
that $\|\tilde z_i\|=1$.  
\end{proof}

\noindent 
We are now prepared to give examples of weakly $(\omega,2)$-divisible elements which are not properly infinite.

\begin{example}
Let $X=B_\infty\times B_3$ and consider the Hilbert module 
$H=\pi_1^*(E)\oplus \pi_2^*(F)$,
described in the statement of the previous proposition. We have shown that $[C(X)]\nleq [H]$. In particular,
$[H]$ is not properly infinite (since it is full). Let us show that $[H]$ is weakly $(\omega,2)$-divisible. Consider the open sets $U :=B_\infty \times B_3^+$ and $V :=B_\infty \times B_3^-$,
where $B_3^+$ and $B_3^-$ are (open) upper and lower hemispheres of $B_3$ that together cover $B_3$.
We claim that $l_2(U)\hookrightarrow HC_0(U)$ and $l_2(V)\hookrightarrow HC_0(V)$. Indeed, 
\begin{align*}
HC_0(U) &=\pi_1^*(E)C_0(U)\oplus \pi_2^*(F)C_0(U)\\
&=\pi_1^*(E)C_0(U)\oplus \pi_2^*(FC_0(B_3^+)).
\end{align*}
But $FC_0(B_3^+)\cong C_0(B_3^+)\oplus C_0(B_3^+\backslash S_2)$.
Therefore,
\begin{align*}
HC_0(U) &=\pi_1^*(E)C_0(U)\oplus C_0(U)\oplus F'\\
&=\pi_1^*(E\oplus C(B_\infty))C_0(U)\oplus F'.
\end{align*}
But $l_2(B_\infty)\hookrightarrow D=E\oplus C(B_\infty)$.
Thus, $l_2(U)\hookrightarrow HC_0(U)$. Symmetrically, we have that 
$l_2(V)\hookrightarrow HC_0(V)$. It follows that $[HC_0(U)]$ and $[HC_0(V)]$
are properly infinite, and $[H]\leq [HC_0(U)]+[HC_0(V)]$. Thus,
$[H]$ is weakly $(\omega,2)$-divisible.
\end{example}

\begin{remark}
The previous example answers a question posed in \cite[Question 3.10]{KirRor:pi}: 
If $a$ and $b$ are properly infinite positive elements, is $a+b$ properly infinite?
In the language of Hilbert modules, this question asks whether $H$ is properly infinite if $H=\overline{H_1+H_2}$, and $H_1,H_2\subseteq H$ are properly infinite submodules of $H$.
We obtain a counterexample taking $H$ as in the previous example, $H_1=HC_0(U)$ and $H_2=HC_0(V)$.
\end{remark}

\begin{example}
In this example we answer (in the negative) the following question, posed in \cite[Question 3.4]{KirRor:pi}: 
if $[H]$ is properly infinite, is the unit of $B(H)$ a properly infinite projection?
Let $X=B_\infty \times (B_3)^{\infty}$ and consider the Hilbert $C(X)$-module
\[
H=C(X)\oplus \pi_1^*(E)\oplus \bigoplus_{i=2}^\infty \pi_i^*(F).
\] 
The module $C(X)\oplus \pi_1^*(E)$ is the pull back along $\pi_1$ of the Dixmier-Douady module $D$. Since $l_2(C(B_\infty))$ embeds in $D$, $l_2(C(X))$ embeds in $C(X)\oplus \pi_1^*(E)$. Thus,  $[H]$ is properly infinite. Also, the direct sum of the module
$\bigoplus_{i=2}^\infty \pi_i^*(F)$ with itself gives $l_2(C(X))$ (because $F\oplus F$ contains $C(B_3)$ as a direct summand). Therefore, $H\oplus H\cong l_2(C(X))$. 
However, $H$ is not  isomorphic to $l_2(C(X))$, because  every section of $\pi_1^*(E)\oplus \bigoplus_{i=2}^\infty \pi_i^*(F)$ vanishes, and so adding the trivial rank 1 module to it cannot yield the trivial Hilbert module $l_2(C(X))$ (see the proof of $D\ncong l_2(C(B_\infty))$ in \cite[Proposition 19]{DixDou:champs}). It follows that $H\oplus H$ is not a direct summand of $H$, i.e., the unit of $B(H)$ is not properly infinite. 
\end{example}

{\small{
\bibliographystyle{amsalpha}
\bibliography{operator-2}}}

\providecommand{\bysame}{\leavevmode\hbox to3em{\hrulefill}\thinspace}
\providecommand{\MR}{\relax\ifhmode\unskip\space\fi MR }
\providecommand{\MRhref}[2]{%
  \href{http://www.ams.org/mathscinet-getitem?mr=#1}{#2}
}
\providecommand{\href}[2]{#2}
\begin{thebibliography}{DHTW09}

\bibitem[BK04a]{BlanKir:Glimm}
E.~Blanchard and E.~Kirchberg, \emph{{Global Glimm halving for $C^*$-bundles}},
  J. Operator Theory \textbf{52} (2004), no.~2, 385--420.

\bibitem[BK04b]{BlanKir:pi3}
\bysame, \emph{{Non-simple purely infinite $C^*$-algebras: the Hausdorff
  case}}, J. Funct. Anal. \textbf{207} (2004), no.~2, 461--513.

\bibitem[BPT08]{BrownPerToms}
N.~Brown, F.~Perera, and A.~S. Toms, \emph{{The {C}untz semigroup, the
  {E}lliott conjecture, and dimension functions on {$C^*$}-algebras}}, J. Reine
  Angew. Math. \textbf{621} (2008), 191--211.

\bibitem[CEI08]{CoEllIv:cuntz}
K.~T. Coward, G.~A. Elliott, and C.~Ivanescu, \emph{{The {C}untz semigroup as
  an invariant for {$C\sp *$}-algebras}}, J. Reine Angew. Math. \textbf{623}
  (2008), 161--193.

\bibitem[DD63]{DixDou:champs}
J.~Dixmier and A.~Douady, \emph{{Champs continus d'espaces hilbertiens et de
  {$C\sp{\ast} $}-alg{\`e}bres}}, Bull. Soc. Math. France \textbf{91} (1963),
  227--284.

\bibitem[DHTW09]{DadHirTomsWin}
M.~D\u{a}d\u{a}rlat, I.~Hirshberg, A.~Toms, and W.~Winter, \emph{{The Jiang-Su
  Algebra does not always embed}}, Math. Res. Lett. \textbf{16} (2009), no.~1,
  23--26.

\bibitem[DT09]{DadToms:Z}
M.~D\u{a}d\u{a}rlat and A.~Toms, \emph{{$\mathcal{Z}$-stability and infinite
  tensor powers of $C^*$-algebras}}, Adv. Math. \textbf{220} (2009), no.~2,
  341--366.

\bibitem[Dup76]{Dupre:vectorbundle}
M.~J. Dupr{\'e}, \emph{{Classifying Hilbert bundles. II}}, J. Funct. Anal.
  \textbf{22} (1976), no.~3, 295--322.

\bibitem[ER06]{EllRor:Hausdorff}
G.~A. Elliott and M.~R{\o}rdam, \emph{{Perturbation of {H}ausdorff moment
  sequences, and an application to the theory of {$C^*$}-algebras of real rank
  zero}}, Operator {A}lgebras: {T}he {A}bel {S}ymposium 2004, Abel Symp.,
  vol.~1, Springer, Berlin, 2006, pp.~97--115.

\bibitem[ERS]{Elliott2008}
G.~A. Elliott, L.~Robert, and L.~Santiago, \emph{{The cone of lower
  semicontinuous traces on a {C}*-algebra}}, American J. Math., to appear.

\bibitem[HR98]{HjeRor:stable}
J.~Hjelmborg and M.~R{\o}rdam, \emph{{On stability of $C^*$-algebras}}, J.
  Funct. Anal. \textbf{155} (1998), no.~1, 153--170.

\bibitem[HRW07]{HirRorWin:C(X)}
I.~Hirshberg, M.~R{\o}rdam, and W.~Winter, \emph{{$C_0(X)$-algebras, stability
  and strongly self-absorbing $C^*$-algebras}}, Math. Ann. \textbf{339} (2007),
  no.~3, 695--732.

\bibitem[Hus94]{Hus:fibre}
D.~Husemoller, \emph{{Fibre Bundles}}, 3rd. ed., Graduate Texts in Mathematics,
  no.~20, Springer Verlag, New York, 1966, 1994.

\bibitem[Kir06]{Kir:Abel}
E.~Kirchberg, \emph{Central sequences in {$C^*$-algebras} and strongly purely
  infinite {$C^*$-algebras}}, Operator Algebras (Berlin) (S.~Neshveyev C.~Skau
  {O. Bratteli}, ed.), Abel Symp., vol.~1, Springer, 2006, pp.~175--232.

\bibitem[KOS03]{KOS:homogeneity}
A.~Kishimoto, N.~Ozawa, and S.~Sakai, \emph{{Homogeneity of the pure state
  space of a separable {$C^*$}-algebra}}, Canad. Math. Bull. \textbf{46}
  (2003), no.~3, 365--372.

\bibitem[KR00]{KirRor:pi}
E.~Kirchberg and M.~R{\o}rdam, \emph{{Non-simple purely infinite
  $C^*$-algebras}}, American J. Math. \textbf{122} (2000), 637--666.

\bibitem[KR02]{KirRor:pi2}
\bysame, \emph{{Infinite non-simple $C^*$-algebras: absorbing the Cuntz algebra
  $\mathcal{O}\_\infty$}}, Advances in Math. \textbf{167} (2002), no.~2,
  195--264.

\bibitem[MS74]{milnor1974characteristic}
J.W. Milnor and J.D. Stasheff, \emph{Characteristic classes}, no.~76, Princeton
  Univ Pr, 1974.

\bibitem[OPR]{OPR:Cuntz}
E.~Ortega, F.~Perera, and M.~{R{\o}rdam}, \emph{{The Corona Factorization
  property, Stability, and the Cuntz semigroup of a C*-algebra.}}, Int. Math.
  Res. Not. IMRN. To appear.

\bibitem[PR04]{PerRor:AF}
F.~Perera and M.~R{\o}rdam, \emph{{AF-embeddings into $C^*$-algebras of real
  rank zero}}, J. Funct. Anal. \textbf{217} (2004), no.~1, 142--170.

\bibitem[PT07]{PerToms:recasting}
F.~Perera and A.~S. Toms, \emph{Recasting the {E}lliott conjecture}, Math. Ann.
  \textbf{338} (2007), no.~3, 669--702.

\bibitem[Rob11]{Robert2011}
L.~Robert, \emph{{The cone of functionals on the Cuntz semigroup}}, 2011.

\bibitem[R{\o}r97]{Ror:sns}
M.~R{\o}rdam, \emph{{Stability of $C^*$-algebras is not a stable property}},
  Documenta Math. \textbf{2} (1997), 375--386.

\bibitem[R{\o}r03]{Ror:simple}
\bysame, \emph{{A simple $C^*$-algebra with a finite and an infinite
  projection}}, Acta Math. \textbf{191} (2003), 109--142.

\bibitem[R{\o}r04]{Ror:Z}
\bysame, \emph{{The stable and the real rank of $\mathcal{Z}$-absorbing
  $C^*$-algebras}}, Internatinal J. Math. \textbf{15} (2004), no.~10,
  1065--1084.

\bibitem[RW10]{RorWin:Z}
M.~R{\o}rdam and W.~Winter, \emph{{The Jiang-Su algebra revisited}}, J. Reine
  Angew. Math \textbf{642} (2010), 129--155.

\bibitem[Vil98]{Vil:perforation}
J.~Villadsen, \emph{{Simple $C^*$-algebras with perforation}}, J. Funct. Anal.
  \textbf{154} (1998), no.~1, 110--116.

\bibitem[Win]{Win:Z-stable}
W.~Winter, \emph{{Nuclear dimension and $\mathcal{Z}$-stability of pure
  $C^*$-algebras}}, Invent. Math., to appear.

\end{thebibliography}


\vspace{1cm}
\noindent{\sc Department of Mathematical Sciences, University of Copenhagen, Universitetsparken 5, DK-2100 Copenhagen \O, Denmark}

\vspace{.2cm}
\noindent{\sl E-mail address:} {\tt leonel@math.ku.dk}

\vspace{.7cm}

\noindent{\sc Department of Mathematical Sciences, University of Copenhagen, Universitetsparken 5, DK-2100 Copenhagen \O, Denmark}

\vspace{.2cm}

\noindent{\sl E-mail address:} {\tt rordam@math.ku.dk}

\end{document}